\documentclass[a4paper,11pt,reqno]{amsart}

\usepackage[T1]{fontenc}
\usepackage[utf8]{inputenc}
\usepackage[english]{babel}
\usepackage{amsmath,amsthm,amssymb}
\usepackage{fullpage}
\usepackage{esint}
\usepackage{graphicx}
\usepackage{enumerate}
\usepackage{pgfplotstable}
\pgfplotsset{compat = 1.17,
tick label style = {font = \tiny},
legend style = {font = \tiny},
}
\usepackage{hyperref}
\usepackage{fancyhdr}
\advance\footskip0.5cm
\makeatletter
\def\@seccntformat#1{%
  \protect\textup{\protect\@secnumfont
    \ifnum\pdfstrcmp{subsection}{#1}=0 \bfseries\fi
    \csname the#1\endcsname
    \protect\@secnumpunct
  }%
}
\makeatother

\numberwithin{equation}{section}

\newtheorem{theorem}{Theorem}[section]
\newtheorem{proposition}[theorem]{Proposition}
\newtheorem{corollary}[theorem]{Corollary}
\newtheorem{lemma}[theorem]{Lemma}
\newtheorem{algorithm}[theorem]{Algorithm}

\newtheorem{remark}[theorem]{Remark}

\newcommand\0{\boldsymbol{0}}
\newcommand\nn{\boldsymbol{n}}
\newcommand\qq{\boldsymbol{q}}
\newcommand\rr{\boldsymbol{r}}
\renewcommand\tt{\boldsymbol{t}}
\newcommand\uu{\boldsymbol{u}}
\newcommand\CC{\boldsymbol{C}}
\newcommand\HH{\boldsymbol{H}}
\newcommand\LL{\boldsymbol{L}}
\newcommand\MM{\boldsymbol{M}}
\newcommand\NN{\boldsymbol{N}}

\newcommand\WW{\boldsymbol{W}}
\newcommand\eps{\varepsilon}
\newcommand\pphi{\boldsymbol{\phi}}
\newcommand\ppsi{\boldsymbol{\psi}}
\newcommand\PPhi{\boldsymbol{\Phi}}
\newcommand\N{\mathbb{N}}
\newcommand\R{\mathbb{R}}

\newcommand\A{\mathcal{A}}
\newcommand\nodes{\mathcal{N}_h}
\newcommand\dts{\boldsymbol{\mathcal{K}}_h}
\newcommand\poly{\mathcal{P}}
\newcommand\sphere{\mathbb{S}^{d-1}}
\newcommand\mesh{\mathcal{T}_h}
\newcommand\VVV{\mathbf{V}}
\newcommand\grad{\nabla}
\newcommand\Grad{\boldsymbol{\nabla}}
\newcommand{\abs}[1]{\lvert #1 \rvert}

\newcommand{\inner}[3][]{\langle #2,#3 \rangle_{#1}}
\newcommand{\norm}[2][]{\lVert #2 \rVert_{#1}}
\newcommand{\weakto}{\rightharpoonup}
\DeclareMathOperator{\diam}{diam}
\DeclareMathOperator*{\argmin}{arg\,min}
\newcommand\tol{\mathsf{tol}}
\newcommand{\wt}{\widetilde}
\newcommand\cdw{c_{\mathrm{dw}}}
\newcommand\nnu{\boldsymbol{\nu}}

\begin{document}

\title{Gamma-convergent projection-free finite element methods
for nematic liquid crystals: The Ericksen model}
\author{Ricardo H.\ Nochetto}
\address{University of Maryland, Department of Mathematics and Institute for Physical Science and Technology,
College Park, MD 20742, USA}
\email{rhn@umd.edu}
\author{Michele Ruggeri}
\address{TU Wien, Institute of Analysis and Scientific Computing,
1040 Vienna, Austria}
\email{michele.ruggeri@asc.tuwien.ac.at}
\author{Shuo Yang}
\address{University of Maryland, Department of Mathematics,
College Park, MD 20742, USA}
\email{shuoyang@umd.edu}

\date{\today}

\begin{abstract}
The Ericksen model for nematic liquid crystals couples a director field with a scalar degree of orientation variable,
and allows the formation of various defects with finite energy.
We propose a simple but novel finite element approximation
of the problem that can be implemented easily within standard finite element packages.
Our scheme is projection-free and thus circumvents the use of weakly acute meshes,
which are quite restrictive in 3D but are required by recent algorithms for convergence.
We prove stability and $\Gamma$-convergence properties of the new method in the presence of defects.
We also design an effective nested gradient flow algorithm for computing minimizers
that controls the violation of the unit-length constraint of the director.
We present several simulations in 2D and 3D that document the performance of the proposed scheme
and its ability to capture quite intriguing defects.
\end{abstract}

\maketitle

\section{Introduction}

\subsection{Liquid crystals with variable degree of orientation}

Liquid crystals (LCs) are a mesophase between crystalline solid and isotropic liquid. They are a host of numerous potential applications in engineering and science, in particular in materials science \cite{ackerman2015self, araki2006colloidal, blanc2016colloidal}. Nematic LCs are made of rod-like molecules with no positional order that tend to point in a preferred direction. LC materials are thus anisotropic.

We consider the one-constant Ericksen model for nematic LCs with variable degree of orientation~\cite{ericksen1991},
which lies between the Oseen--Frank director model and the Landau--de Gennes $Q$-tensor
model~\cite{de1993physics,virga1994}.
The state of the LC is described in terms of a vector field $\nn$ and a scalar function $s$,
which satisfy the constraints $\abs{\nn}=1$ and $-1/(d-1)<s<1$ for the space dimension $d=2,3$.
The director $\nn$ indicates the preferred orientations of the LC molecules,
while $s$ represents the degree of alignment that the molecules have with respect to $\nn$,
both in the sense of local probabilistic average.
A schematic illustration of their meaning is given in Figure~\ref{fig:order}.
The equilibrium state is given by an admissible pair $(s, \nn)$ that minimizes the Ericksen energy
\begin{equation} \label{eq:Ericksen-energy}
E[s,\nn]
= 
\frac{1}{2} \int_{\Omega} \big( \kappa \abs{\grad s}^2 + s^2\abs{\Grad\nn}^2\big)
+
\int_{\Omega} \psi (s),
\end{equation}
where $\kappa>0$ is constant; the constraint on $s$ is enforced by the double well potential $\psi$.
We refer to~\cite{ambrosio1990,lin1991} for early analysis of the Ericksen model.

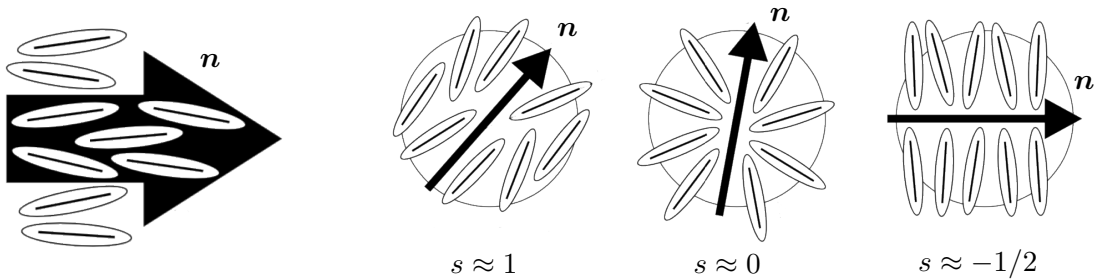
\begin{figure}[htbp]
\centering
\begin{tikzpicture}
\draw (0, 0) node[] {\includegraphics[height=3.5cm]{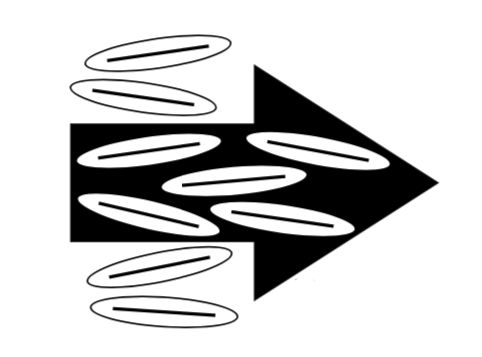}};
\draw (8, 0) node[inner sep=0] {\includegraphics[height=4cm]{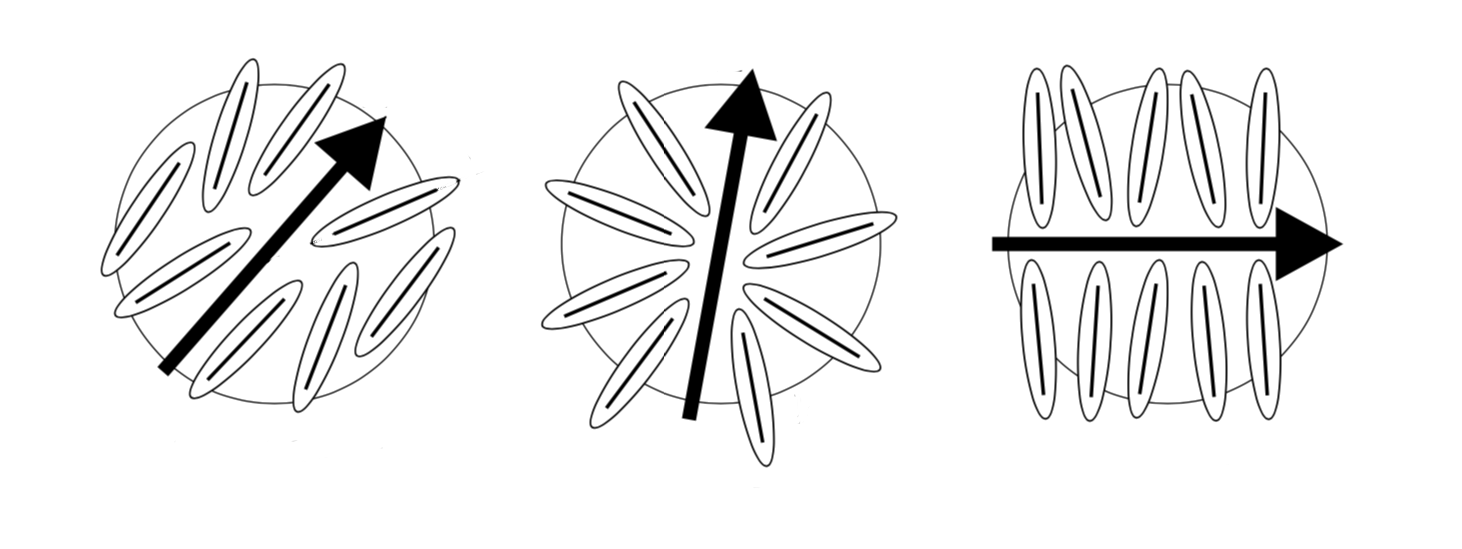}};
\draw (1, 1) node {$\nn$};
\draw (4.6, -1.7) node {$s \approx 1$};
\draw (7.8, -1.7) node {$s \approx 0$};
\draw (11.1, -1.7) node {$s \approx -1/2$};
\draw (5.7, 1.4) node {$\nn$};
\draw (8.5, 1.6) node {$\nn$};
\draw (12.5, 0.7) node {$\nn$};
\end{tikzpicture}
\caption{Schematic illustration of $\nn(x)$ and $s(x)$, in microscopic scale near a fixed $x \in \Omega \subset \R^3$.
Note that $s = 1$ represents the state of perfect alignment in which all molecules in the local ensemble are parallel to $\nn$.
Likewise, $s = -1/2$ represents the state of perpendicular alignment.
The case $s = 0$ corresponds to a defect in the LC material, an isotropic distribution of molecules in the local ensemble that do not lie along any preferred direction.}
\label{fig:order}
\end{figure}

If $s$ can be approximated by a nonvanishing constant, then the energy~\eqref{eq:Ericksen-energy}
reduces to the Oseen--Frank energy $E[\nn] \propto \int_\Omega\abs{\Grad\nn}^2$,
whose minimizers are harmonic maps and have been extensively studied, e.g., in~\cite{schoen1982regularity,brezis1986harmonic}. However, the simpler Oseen--Frank model has severe limitations in capturing defects:
It only admits point defects with finite energy for $d=3$.
In contrast, the Ericksen model~\eqref{eq:Ericksen-energy} allows for $\nn \notin \HH^1(\Omega)$
and
compensates blow-up of $\Grad\nn$
by letting $s$ vanish, which is the mechanism for the formation of a variety of line and surface defects for $d=2,3$.
This physical process leads to a \emph{degenerate} Euler--Lagrange equation for $\nn$ that poses serious difficulties
to formulate mathematically sound algorithms to approximate~\eqref{eq:Ericksen-energy} and study their convergence.

\subsection{Numerical analysis of the Ericksen model}

Several numerical methods for the Oseen--Frank model have been proposed~\cite{lin1989relaxation,alouges1997,bartels2010numerical,bartels2016}.
Finite element methods (FEMs) for the Ericksen model are designed
in~\cite{BFP2006,nwz2017,nochetto2018ericksen,walker2020,carter2020domain};
see also the recent review~\cite{bw2021}.
In contrast to \cite{BFP2006}, a fundamental structure of \eqref{eq:Ericksen-energy} is exploited in~\cite{nwz2017,nochetto2018ericksen}
to design and analyze FEMs that handle the inherent degeneracy of \eqref{eq:Ericksen-energy}
without regularization and enforce the constraint $\abs{\nn}=1$ robustly.
Stability and convergence properties via $\Gamma$-convergence are proved in~\cite{nwz2017,nochetto2018ericksen},
pioneering results in this setting. They hinge on a clever discrete energy that mimics the structure of~\eqref{eq:Ericksen-energy}
discretely but, unfortunately, is cumbersome to implement in standard software packages and requires weakly acute meshes.
The latter ensures that the projection of discrete director fields onto the unit sphere is energy decreasing,
and thus compatible with the quasi-gradient flow, but is quite restrictive and difficult to implement for $d=3$
and domains with nontrivial topology.

\subsection{Contributions}

In this work, we propose a projection-free FEM that avoids dealing with weakly acute meshes. Without the projection step, the unit-length constraint $\abs{\nn}=1$ is no longer satisfied exactly but instead is relaxed at each step of our iterative solver, a nested gradient flow. The latter guarantees control of the violation of $\abs{\nn}=1$ and asymptotic enforcement of it. We summarize the chief novelties and advantages of our approach as follows.

\begin{enumerate}[$\bullet$]
\item
\emph{Shape-regular meshes}.
Partitions of $\Omega$ are assumed to be only shape-regular,
which allows for the use of software with general mesh generators such as Netgen~\cite{schoberl2017netgen}.
Avoiding weakly acute meshes is important in 3D to deal with interesting but nontrivial geometries as documented in Section~\ref{sec:numerics}.
An earlier work achieving this goal is~\cite{walker2020},
which presents a mass-lumped FEM with a consistent stabilization term involving $s^2 \Grad\nn^\top\nn$
for the generalized Ericksen energy.

\smallskip
\item
\emph{Standard algorithm}.
Our novel discretization of~\eqref{eq:Ericksen-energy} is straightforward,
requires no stabilization,
and is easy to implement in standard software packages
such as NGSolve~\cite{schoberl2017netgen}.
In contrast to \cite{nwz2017,nochetto2018ericksen}, our FEM does no longer exploit the structure of~\eqref{eq:Ericksen-energy},
but its analysis does.

\smallskip
\item
\emph{Linear solver}.
We propose a nested gradient flow that,
despite the nonlinear nature of the problem,
is \emph{fully linear}
to compute minimizers.
The inner loop to advance the director $\nn$ for fixed degree of orientation $s$ is allowed to subiterate.
This turns out to induce an acceleration mechanism for the computation and motion of defects.
For a recent acceleration technique based on a domain decomposition approach,
we refer to~\cite{carter2020domain}.

\smallskip
\item
\emph{$\Gamma$-convergence}.
The analysis of our FEM hinges heavily on the underlying structure of~\eqref{eq:Ericksen-energy},
which is fully discussed in Section~\ref{sec:model}
and relies on the notion of $L^2$-gradient on $\nn$~\cite[Theorem~6.2]{eg2015};
see Proposition~\ref{prop:L2diff} below.
Such a notion was already used in~\cite{BNW2020} in the context of the uniaxial Q-tensor LC model.
We prove stability and $\Gamma$-convergence.
Our results are similar to those in~\cite{nwz2017,nochetto2018ericksen,walker2020}, but the use of the discrete structure is new.

\smallskip
\item
\emph{Numerical experiments}.
We present several simulations in Section \ref{sec:numerics}. Some are meant to compare the new algorithm with the existing literature in terms of performance and ability to capture defects. Other experiments explore 3D intriguing configurations such as the propeller defect and challenging variations of the Saturn ring defect.

\smallskip
\item
\emph{Boundary conditions}.
Since we do not impose the unit-length constraint $\abs{\nn}=1$, the treatment of boundary data can be simplified and their properties weakened. This affects the regularization procedure for the lim-sup property and the possible presence of defects at the boundary of $\Omega$. We do not explore these issues in this paper but rather in future extension to the $Q$-tensor model.

\end{enumerate}
 
\subsection{Outline}

The remainder of this work is organized as follows.
In the next short subsection, we collect some general notation used throughout the paper.
In Section~\ref{sec:model}, we describe the Ericksen model for LCs with variable degree of orientation and discuss its key structure.
In Section~\ref{sec:discretization}, we introduce our discretization of the model and state
our $\Gamma$-convergence result.
In Section~\ref{sec:algorithm},
we present our iterative scheme for the computation of discrete local minimizers.
In Section~\ref{sec:numerics}, we show numerical experiments illustrating effectiveness and efficiency of our method, as well as its flexibility to deal with complex defects in 3D.
We postpone the proofs of most results to the Section~\ref{sec:proofs}.

\subsection{General notation}

We denote by $\N = \{ 1,2,\dots\}$ the set of natural numbers
and set $\N_0 := \N \cup \{ 0 \}$.
For $d=2,3$, we denote the unit sphere in $\R^d$ by $\sphere = \{ x \in \R^d : \abs{x} = 1 \}$.
We denote by $B_r(x)$ the ball of radius $r>0$ centered at $x \in \R^d$.
For (spaces of) vector- or matrix-valued functions, we use bold letters,
e.g., for a generic domain $\Omega \subset \R^d$,
we denote both $L^2(\Omega;\R^d)$ and $L^2(\Omega;\R^{d \times d})$ by $\LL^2(\Omega)$.
We denote by $\inner{\cdot}{\cdot}$
both the scalar product of $\LL^2(\Omega)$ and the duality pairing between $\HH^1(\Omega)$ and its dual,
with the ambiguity being resolved by the arguments.
We use the notation $\lesssim$ to denote
\emph{smaller than or equal to up to a multiplicative constant},
i.e., we write $A \lesssim B$ if there exists a constant $c>0$,
which is clear from the context and always independent of the discretization parameters,
such that $A \leq c B$.

\section{Problem formulation} \label{sec:model}

Let $\Omega \subset \R^d$ ($d=2,3$) be a bounded Lipschitz domain.
In the Ericksen model,
the state of the LC is described in terms of
a unit-length vector field $\nn : \Omega \to \sphere$
and a scalar function $s : \Omega \to (-1/(d-1),1)$.
Equilibrium configurations of the LC are
minimizers of the energy $E[s,\nn] = E_1[s,\nn] + E_2[s]$ in \eqref{eq:Ericksen-energy}, where
\begin{equation} \label{eq:energy}
  E_1[s,\nn] := \frac{1}{2} \int_{\Omega} \big( \kappa \abs{\grad s}^2 + s^2\abs{\Grad\nn}^2\big),
  \qquad
 E_2[s] := \int_{\Omega} \psi (s).
\end{equation}
The double well potential
$\psi : (-1/(d-1),1) \to \R_{\ge 0}$ satisfies the following properties~\cite{ericksen1991}:
\begin{itemize}
\item $\psi \in C^2(-1/(d-1),1)$,
\item $\lim_{s \to 1^-} \psi(s) = + \infty = \lim_{s \to {-1/(d-1)}^+} \psi(s)$,
\item $\psi(0) > \psi(s^*) = \min_{s \in (-1/(d-1),1)} \psi(s) = 0$ for some $s^* \in (0,1)$,
\item $\psi'(0) = 0$.
\end{itemize}
In~\eqref{eq:energy},
$E_1[s,\nn]$ is the one-constant approximation of the elastic energy
proposed in~\cite{ericksen1991},
while $E_2[s]$ is a potential energy which confines the variable $s$ within the physically admissible interval $(-1/(d-1),1)$.
The presence of the weight $s^2$ in the second term of $E_1[s,\nn]$
allows for blow-up of $\Grad\nn$, namely $\nn \notin \HH^1(\Omega)$, in the \emph{singular set}
\begin{equation} \label{eq:singular}
\Sigma := \{ x \in \Omega : s(x) = 0\},
\end{equation}
where defects may occur.

To complete the setting,
we define the set of admissible functions where we seek minimizers of~\eqref{eq:energy}.
Note that, allowing for a director $\nn \notin \HH^1(\Omega)$,
one encounters at least two difficulties:
On the one hand,
it is not clear how to interpret the gradient of $\nn$ appearing in $E_1[s,\nn]$.
On the other hand,
the trace of $\nn$ on the boundary of $\Omega$ is not well-defined,
so that one cannot impose Dirichlet conditions on $\nn$ in the standard way.
To cope with these problems,
following~\cite{ambrosio1990,lin1991}, we introduce the auxiliary variable $\uu = s \nn$.
Then, the product rule formally yields that
\begin{equation} \label{eq:productRule}
\Grad \uu
= \nn \otimes \grad s + s \Grad \nn.
\end{equation}
Since $\abs{\nn} = 1$, the identities
$\Grad \nn^\top \nn = \0$
and $\abs{\nn \otimes \grad s} = \abs{\grad s}$ are valid.
It follows that the above decomposition of $\Grad\uu$ is orthogonal,
i.e.,
\begin{equation} \label{eq:decomposition1}
\abs{\Grad \uu}^2
= \abs{\nn \otimes \grad s}^2
+ s^2 \abs{\Grad \nn}^2
= \abs{\grad s}^2
+ s^2 \abs{\Grad \nn}^2.
\end{equation}
In particular, $E_1[s,\nn]$ can be rewritten in terms of $s$ and $\uu = s \nn$ as
\begin{equation}\label{eq:equiv-energy}
E_1[s,\nn]
= \wt E_1[s,\uu]
= \frac{1}{2} \int_{\Omega} \big( (\kappa - 1) \abs{\grad s}^2 + \abs{\Grad\uu}^2 \big).
\end{equation}
In the latter,
the degree of orientation and the auxiliary field are decoupled.
In particular, this reveals that, for $(s,\nn)$ such that $E_1[s,\nn] < \infty$,
$\uu = s \nn \in \HH^1(\Omega)$ even though $\nn \notin \HH^1(\Omega)$.

We say that a triple $(s,\nn,\uu)$
satisfies the \emph{structural condition} if
\begin{equation} \label{eq:structural}
- \frac{1}{d-1} < s < 1,
\quad
\abs{\nn} = 1,
\quad
\text{and}
\quad
\uu = s \nn
\quad
\text{a.e.\ in } \Omega.
\end{equation}
In view of the above discussion,
we are therefore led to consider the following admissible class:
\begin{equation}\label{eq:admissible}
\A
:=\big\{(s,\nn,\uu) \in H^1(\Omega) \times \LL^\infty(\Omega) \times \HH^1(\Omega):
(s,\nn,\uu) \text{ satisfies } \eqref{eq:structural}
\big\}.
\end{equation}

For triples $(s,\nn,\uu) \in \A$,
it is possible to characterize the gradient of $\nn$ occurring in $E_1[s,\nn]$
using a weaker notion of differentiability.
To this end, we recall the following definition \cite[Theorem~6.2]{eg2015}:
We say that $\nn$ is $L^2$-differentiable at $x \in \Omega$,
and we denote its $L^2$-gradient at $x$ by $\Grad\nn(x)$, if
\begin{equation*}
\fint_{B_r(x)} \abs{\nn(y) - \nn(x) - \Grad \nn(x)(y-x)}^2 \, \mathrm{d}y = o(r^2)
\quad \text{as } r \to 0.
\end{equation*}
It is well-known that the notion of $L^2$-differentiability
is weaker than the existence of a $L^2$-integrable weak gradient,
in the sense that every $H^1$-function is $L^2$-differentiable almost everywhere
and its $L^2$-gradient coincides with the weak gradient;
see, e.g., \cite[Theorem~6.2]{eg2015}.

In the following proposition,
we establish that if $(s,\nn,\uu) \in \A$, then
$\nn$ is $L^2$-differentiable and the decomposition~\eqref{eq:decomposition1} holds
almost everywhere outside of the singular set $\Sigma$ in \eqref{eq:singular}.
Its proof will be presented in Section~\ref{sec:L2diff}.

\begin{proposition}[orthogonal decomposition] \label{prop:L2diff}
Let $(s,\nn,\uu) \in \A$.
Then, $\nn$ is $L^2$-differentiable a.e.\ in $\Omega\setminus\Sigma$.
In particular,
its $L^2$-gradient is given by
\begin{equation} \label{eq:L2gradient_formula}
\Grad \nn
= s^{-1} (\Grad\uu - \nn\otimes\grad s)
\quad
\text{a.e.\ in } \Omega\setminus\Sigma.
\end{equation}
Moreover, the following identity holds
\begin{equation} \label{eq:structural2}
\abs{\Grad \uu}^2
= \abs{\grad s}^2
+ s^2 \abs{\Grad \nn}^2
\quad
\text{a.e.\ in } \Omega\setminus\Sigma.
\end{equation}
\end{proposition}

This allows us to give a precise meaning to $E_1[s,\nn]$ in \eqref{eq:energy}.
Depending on the context, we interpret $\Grad\nn$ in the sense of $L^2$-gradient
in $\Omega\setminus\Sigma$
and $\int_\Sigma s^2|\Grad\nn|^2=0$, or we alternatively replace $\Omega$
by $\Omega\setminus\Sigma$ as domain of integration or even use the representation
$\wt{E}_1[s,\uu]$ of \eqref{eq:equiv-energy}.

Turning to boundary conditions,
let $\Gamma_D \subseteq \partial\Omega$ 
be a relatively open subset of the boundary
such that $\abs{\Gamma_D}>0$,
where we aim to impose Dirichlet boundary conditions.
These, in the context of LCs, are usually referred to as \emph{strong anchoring} conditions.
To this end, given a triple $(g, \qq, \rr) \in W^{1,\infty}(\R^3) \times \LL^{\infty}(\R^3) \times \WW^{1,\infty}(\R^3)$
satisfying the structural condition~\eqref{eq:structural},
we consider the following restricted admissible class that incorporates boundary conditions:
\begin{equation}\label{eq:admissibleBC}
\A(g,\rr)
:=\big\{(s,\nn,\uu) \in \A:
s \vert_{\Gamma_D} = g \vert_{\Gamma_D}
\text{ and }
\uu \vert_{\Gamma_D}= \rr \vert_{\Gamma_D}
\big\}.
\end{equation}
Overall, we are interested in the following constrained minimization problem:
Find $(s^*,\nn^*,\uu^*)\in\A(g,\rr)$ such that
\begin{equation} \label{eq:minimization}
(s^*,\nn^*,\uu^*) = \argmin_{(s,\nn,\uu)\in\A(g,\rr)}E[s,\nn].
\end{equation}

To conclude this section,
let $\delta_0 > 0$ be sufficiently small.
Some of our results below will require the following technical assumptions
on the Dirichlet data, namely
\begin{gather}\label{eq:DataAwayFromBoundaryS}
- \frac{1}{d-1} + \delta_0
\le g(x)
\le 1 - \delta_0
\quad
\text{for all } x \in \R^d,
\\ \label{eq:DataAwayFrom0}
g \ge \delta_0
\quad
\text{on } \Gamma_D,
\end{gather}
and on the double well potential, namely
\begin{equation} \label{eq:DoubleWellAwayFromBoundaryS}
\begin{aligned}
\psi(s) &\ge \psi(1-\delta_0)
&&\text{for all } s \ge 1-\delta_0, \\
\psi(s) &\ge \psi\left(- \frac{1}{d-1} + \delta_0\right)
&&\text{for all } s \le - \frac{1}{d-1} + \delta_0,
\end{aligned}
\end{equation}
and $\psi$ in monotone in $(-1/(d-1),-1/(d-1)+\delta_0)$ and in $(1-\delta_0,1)$.
Note that~\eqref{eq:DataAwayFrom0} implies that $\qq = g^{-1}\rr$ is $\WW^{1,\infty}$
in a neighborhood of $\Gamma_D$
and hence $\nn$ is $\HH^1$
in a neighborhood of $\Gamma_D$,
so that in this case one can impose the Dirichlet conditions
$\nn \vert_{\Gamma_D}= \qq \vert_{\Gamma_D}$
directly on $\nn$.
Finally, the property~\eqref{eq:DoubleWellAwayFromBoundaryS}
is consistent with the fact that $\psi(s) \to + \infty$ as $s \to - 1/(d-1)$ and $s \to 1$.

\section{$\Gamma$-convergent finite element discretization} \label{sec:discretization}

We assume $\Omega$ be a polytopal domain
and consider a shape-regular family $\{ \mesh \}$
of simplicial meshes of $\Omega$
parametrized by the mesh size $h = \max_{K \in \mesh} h_K$,
where $h_K = \diam(K)$.
We denote by $\nodes$ the set of vertices of $\mesh$.
For any $K \in \mesh$, we denote by $\poly^1(K)$
the space of first-order polynomials on $K$.
We consider the space of $\mesh$-piecewise affine and globally continuous functions
\begin{equation*}
V_h
:=
\left\{v_h \in C^0(\overline{\Omega}): v_h \vert_K \in \poly^1(K) \text{ for all } K \in \mesh \right\}.
\end{equation*}
Let $\VVV_h:= (V_h)^d$ be the corresponding space of vector-valued polynomials.
We denote by $I_h$ both the nodal interpolant $I_h: C^0(\overline{\Omega}) \to V_h$
and its vector-valued counterpart $I_h: \CC^0(\overline{\Omega}) \to \VVV_h$.

For $s_h \in V_h$ and $\nn_h \in \VVV_h$, let the discrete energy be
$E^h[s_h,\nn_h] = E^h_1[s_h,\nn_h] + E^h_2[s_h]$ with
\begin{equation} \label{eq:energy_h}
E^h_1[s_h,\nn_h]
:= \frac{1}{2} \int_{\Omega} \big( \kappa \abs{\nn_h \otimes \grad s_h}^2
+ s_h^2\abs{\Grad\nn_h}^2\big),
\qquad
E^h_2[s_h] := \int_{\Omega} \psi (s_h).
\end{equation}
Note that $E^h$ is consistent,
in the sense that $E^h[s,\nn] = E[s,\nn]$ if $(s,\nn,\uu) \in \A(g,\rr)$.

We say that a triple $(s_h,\nn_h,\uu_h) \in V_h \times \VVV_h \times \VVV_h$
satisfies the \emph{discrete structural condition} if
\begin{equation} \label{eq:structural_h}
- \frac{1}{d-1} < s_h(z) < 1,
\quad
\abs{\nn(z)} \ge 1,
\quad
\text{and}
\quad
\uu_h(z) = s_h(z) \nn_h(z)
\quad
\text{for all } z \in \nodes.
\end{equation}
In~\eqref{eq:structural_h},
the requirements prescribed by the continuous structural condition~\eqref{eq:structural}
are imposed only at the vertices of the mesh, which is practical.
Moreover, the unit-length constraint for the director is relaxed,
since $\nn_h$ may attain also values outside of the unit sphere.

Let $\eps>0$, $g_h = I_h[g]$, and $\rr_h = I_h[\rr]$.
We consider the following discrete minimization problem:
Find $(s_h^*,\nn_h^*,\uu_h^*)\in\A_{h,\eps}(g_h,\rr_h)$ such that
\begin{equation}\label{eq:discrete-minimization}
(s_h^*,\nn_h^*,\uu_h^*) = \argmin_{(s_h,\nn_h,\uu_h)\in\A_{h,\eps}(g_h,\rr_h)}E_h[s_h,\nn_h],
\end{equation}
where the discrete restricted admissible class is defined as 
\begin{multline}\label{eq:admissible_h}
\A_{h,\eps}(g_h,\rr_h)
:=
\big\{ (s_h,\nn_h,\uu_h)\in V_h\times\VVV_h \times\VVV_h: \\
(s_h,\nn_h,\uu_h) \text{ satisfies }\eqref{eq:structural_h}, \
\norm[L^1(\Omega)]{I_h \big[\abs{\nn_h}^2 \big]-1} \le \eps, \\
s_h(z) = g_h(z), \text{ and }
u_h(z) = r_h(z) \text{ for all } z \in\nodes\cap\Gamma_D \big\}.
\end{multline}

In the following theorem,
we show that the discrete energy~\eqref{eq:energy_h} converges towards the continuous one~\eqref{eq:energy}
in the sense of $\Gamma$-convergence.

\begin{theorem}[$\Gamma$-convergence] \label{thm:gamma_convergence}
Suppose that $\eps \to 0$ as $h \to 0$.
Then, the following two properties are satisfied:
\begin{enumerate}[\rm(i)]
\item
Lim-sup inequality (consistency):
Let $\Gamma_D = \partial\Omega$.
Let the assumptions~\eqref{eq:DataAwayFromBoundaryS}--\eqref{eq:DoubleWellAwayFromBoundaryS} hold.
If $(s,\nn,\uu)\in\A(g,\rr)$,
then there exists a sequence $\{ (s_h,\nn_h,\uu_h) \} \subset \A_{h,\eps}(g_h,\rr_h)$
such that
$s_h\to s$ in $H^1(\Omega)$,
$\nn_h\to\nn$ in $\LL^2(\Omega\setminus\Sigma)$,
$\uu_h\to \uu$ in $\HH^1(\Omega)$,
as $h \to 0$, and
\begin{equation} \label{eq:limsup_inequality}
E[s,\nn]\ge\limsup_{h \to 0} E^h[s_h,\nn_h].
\end{equation}
\item
Lim-inf inequality (stability):
Let $\{ (s_h, \nn_h, \uu_h) \} \subset \A_{h,\eps}(g_h,\rr_h)$
be a sequence such that
$E^h[s_h,\nn_h] \le C$ and
$\norm[\LL^{\infty}(\Omega)]{\nn_h} \le C$,
where $C \ge 1$ is a constant independent of $h$.
Then,
there exist $(s,\nn,\uu)\in\A(g,\rr)$ and a subsequence of $\{ (s_h,\nn_h, \uu_h) \}$ (not relabeled)
such that $s_h \weakto s$ in $H^1(\Omega)$,
$\nn_h\to\nn$ in $\LL^2(\Omega\setminus\Sigma)$,
$\uu_h \weakto \uu$ in $\HH^1(\Omega)$
as $h \to 0$, and
\begin{equation} \label{eq:liminf_inequality}
E[s,\nn] \le \liminf_{h \to 0}E^h[s_h,\nn_h].
\end{equation}
\end{enumerate}
\end{theorem}

The proof of Theorem~\ref{thm:gamma_convergence} is deferred to Sections~\ref{sec:limsup}--\ref{sec:liminf}.
The properties established in Theorem~\ref{thm:gamma_convergence}
are slight variations of the properties required by the standard definition of $\Gamma$-convergence;
see, e.g., \cite[Definition~1.5]{braides2002}.
However,
they still allow
to prove the convergence of discrete global minimizers.

\begin{corollary}[convergence of discrete global minimizers] \label{cor:minimizers}
Let $\Gamma_D = \partial\Omega$ and suppose that
the assumptions~\eqref{eq:DataAwayFromBoundaryS}--\eqref{eq:DoubleWellAwayFromBoundaryS} hold.
Let $\{ (s_h,\nn_h, \uu_h) \} \subset \A_{h,\eps}(g_h,\rr_h)$
be a sequence of global minimizers of the discrete energy~\eqref{eq:energy_h}
such that $\norm[\LL^{\infty}(\Omega)]{\nn_h} \le C$,
where $C \ge 1$ is a constant independent of $h$.
Then, every cluster point $(s,\nn,\uu)$ belongs to $\A(g,\rr)$
and is a global minimizer of the continuous energy~\eqref{eq:energy}.
\end{corollary}

\section{Computation of discrete local minimizers} \label{sec:algorithm}

In this section,
we propose an effective algorithm to compute discrete local minimizers of~\eqref{eq:energy_h}.
The method is based on a discretization of the energy-decreasing dynamics driven
by the system of gradient flows
\begin{align*}
\partial_t \nn + \delta_{\nn} E^h[s,\nn] & = 0,\\
\partial_t s + \delta_{s} E^h[s,\nn] & = 0,
\end{align*}
where $\delta_{\nn} E^h[s,\nn]$ and $\delta_{s} E^h[s,\nn]$ denote
the G\^ateaux derivatives of the energy with respect to the order parameters, i.e.,
\begin{align*}
\big\langle \delta_{\nn} E^h[s,\nn] , \pphi \big\rangle
& = \big\langle \delta_{\nn} E_1^h[s,\nn] , \pphi \big\rangle
= \kappa \inner{\nn \otimes \grad s}{\pphi \otimes \grad s}
+ \inner{s \Grad \nn}{s \Grad \pphi}, \\
\big\langle \delta_{s} E^h[s,\nn] , w \big\rangle
& = \big\langle \delta_{s} E_1^h[s,\nn] , w \big\rangle
+ \big\langle \delta_{s} E_2^h[s,\nn] , w \big\rangle \\
& = \kappa \inner{\nn \otimes \grad s}{\nn \otimes \grad w}
+ \inner{s \Grad \nn}{w \Grad \nn}
+ \inner{\psi'(s)}{w}.
\end{align*}
Let us introduce the ingredients of the scheme.
First, let
\begin{equation*}
V_{h,D} := \{ v_h \in V_h : ~ v_h(z) = 0 \text{ for all } z \in \nodes \cap \Gamma_D \}
\quad
\text{and}
\quad
\VVV_{h,D} := (V_{h,D})^d
\end{equation*}
be the spaces of discrete functions satisfying homogeneous Dirichlet conditions on $\Gamma_D$.
Given $\nn_h \in \VVV_h$, we consider the subspace of $\VVV_{h,D}$
consisting of all discrete functions
with nodal values
orthogonal to those of $\nn_h$ at all vertices:
\begin{equation*}
\dts[\nn_h] := \left\{\pphi_h \in \VVV_{h,D}: ~\nn_h(z)\cdot\pphi_h(z) = 0 \text{ for all } z \in \mathcal{N}_h \right\}.
\end{equation*}
For the treatment of the double well potential,
we follow a convex splitting approach (see, e.g., \cite{wwl2009}):
we assume the splitting $\psi = \psi_c - \psi_e$,
where $\psi_c$ and $\psi_e$ are both convex and $\psi_c$ is quadratic.

The \emph{time} discretization
of the gradient flow for the director and the degree of orientation
are based on the constant time-step sizes $\tau_{\nn}>0$ and $\tau_s>0$, respectively.
Moreover, we consider the difference quotient $d_t s_h^{i+1}:= (s_h^{i+1}-s_h^i) / \tau_s$.

In the following algorithm,
we state the proposed numerical scheme for the computation of
discrete local minimizers of~\eqref{eq:energy_h}.
We assume that assumption~\eqref{eq:DataAwayFrom0} is satisfied so that imposing
Dirichlet boundary conditions directly for the director is allowed.
Let $\tol>0$ denote a tolerance.

\begin{algorithm}[alternating direction discrete gradient flow] \label{alg:gradient_flow}
\underline{Input}:
$s_h^0 \in V_h$, $\nn_h^0 \in \VVV_h$
such that
$\abs{\nn_h^0(z)} = 1$ for all $z \in \nodes$,
$\nn_h^0(z) = \rr_h(z) / g_h(z)$ and $s_h^0(z)=g_h(z)$ for all $z\in\nodes\cap\Gamma_D$.\\
\underline{Outer loop}:
For all $i \in \N_0$, iterate {\rm(i)}--{\rm(ii)}:
\begin{itemize}
\item[\rm(i)] \underline{Inner loop}:
Given $(\nn_h^i,s_h^i)$, let $\nn_h^{i,0}=\nn_h^i$.
For all $\ell \in \N_0$, iterate {\rm(i-a)}--{\rm(i-b)}:
\begin{itemize}
\item[\rm(i-a)]
Compute $\tt_h^{i,\ell} \in \dts\big[\nn_h^{i,\ell}\big]$ such that
\begin{equation} \label{eq:gradientflow1}
\begin{aligned}
 \inner[*]{\tt_h^{i,\ell}}{\pphi_h}
& + \tau_{\nn} \, \kappa \inner{ \tt_h^{i,\ell}  \otimes \grad s_h^i}{\pphi_h \otimes \grad s_h^i}
+ \tau_{\nn} \inner{s_h^i \Grad \tt_h^{i,\ell}}{s_h^i \Grad \pphi_h} \\
& = - \kappa \inner{ \nn_h^{i,\ell}  \otimes \grad s_h^i}{\pphi_h \otimes \grad s_h^i}
- \inner{s_h^i \Grad \nn_h^{i,\ell}}{s_h^i \Grad \pphi_h}
\end{aligned}
\end{equation}
for all $\pphi_h \in \dts\big[\nn_h^{i,\ell}\big]$;
\item[\rm(i-b)] Update $\nn^{i,\ell+1}_h:=\nn^{i,\ell}_h + \tau_{\nn} \, \tt_h^{i,\ell}$;
\end{itemize}
until
\begin{equation} \label{eq:gradientflow2_stopping}
\big\lvert E_1^h[s_h^i,\nn^{i,\ell+1}_h]-E_1^h[s_h^i,\nn^{i,\ell}_h] \big\rvert < \tol.
\end{equation}
If $\ell_i \in \N_0$ denotes the smallest integer for which the stopping criterion~\eqref{eq:gradientflow2_stopping} is satisfied,
define $\nn_h^{i+1}:=\nn^{i,\ell_i+1}_h$.

\item[{\rm(ii)}] Compute $s_h^{i+1} \in V_h$ such that $s_h^{i+1}(z)=g_h(z)$ for all $z\in\nodes\cap\Gamma_D$ and
\begin{equation} \label{eq:gradientflow3}
\begin{aligned}
  \inner{d_t s_h^{i+1}}{w_h}
& + \kappa \inner{ \nn_h^{i+1}  \otimes \grad s_h^{i+1}}{\nn_h^{i+1} \otimes \grad w_h} \\
& + \inner{s_h^{i+1} \Grad \nn_h^{i+1}}{w_h \Grad \nn_h^{i+1}}
  + \inner{\psi_c'(s_h^{i+1})}{w_h}
= \inner{\psi_e'(s_h^i)}{w_h}
\end{aligned}
\end{equation}
for all $w_h \in V_{h,D}$.
\end{itemize} 
\underline{Output}:
Sequence of approximations $\left\{ (s_h^i,\nn_h^i)\right\}_{i \in \N_0}$.
\end{algorithm}

In Algorithm~\ref{alg:gradient_flow}, $\inner[*]{\cdot}{\cdot}$ denotes the scalar product of the metric used in
the discrete gradient flow~\eqref{eq:gradientflow1} for the director.
In this work, we consider the following two choices for $\inner[*]{\cdot}{\cdot}$:
\begin{align}
\inner[*]{\pphi}{\ppsi}
& =
\inner{\pphi}{\ppsi}
&&\text{($L^2$-metric)},
\label{eq:metricL2}
\\
\inner[*]{\pphi}{\ppsi}
& =
\inner{h^{\alpha} \Grad\pphi}{\Grad\ppsi},
\quad
\text{with } 0<\alpha\le2.
&&\text{(weighted $H^1$-metric)},
\label{eq:metricH1_weighted}
\end{align}
Note that in~\eqref{eq:metricH1_weighted} the choice $\alpha=0$ corresponds to a full $H^1$-gradient flow,
which is not appropriate since the director does not belong to $\HH^1(\Omega)$ in general
(e.g., in the presence of defects).
On the other hand, if $\alpha=2$, the resulting metric is equivalent to the $L^2$-metric in~\eqref{eq:metricL2}.
In addition, both~\eqref{eq:gradientflow1} and~\eqref{eq:gradientflow3} are linear symmetric positive definite systems
in the unknowns $\tt_h^{i,\ell}$ and $s_h^{i+1}$.

Although in most of our numerical experiments we will set $\tau_{\nn} = \tau_s$,
we observed that in some situations the flexibility of choosing different time-step sizes
in~\eqref{eq:gradientflow1} and~\eqref{eq:gradientflow3}
is decisive in order to move defects in numerical simulations
(see, e.g., the experiment in Section~\ref{sec:cylinder_ex} below).

In the following proposition,
we prove well-posedness and an energy-decreasing property of Algorithm~\ref{alg:gradient_flow}.

\begin{proposition}[properties of Algorithm~\ref{alg:gradient_flow}] \label{prop:scheme}
Algorithm~\ref{alg:gradient_flow} is well-posed and energy decreasing.
Specifically, for all $i \in \N_0$,
the following assertions hold:
\begin{itemize}
\item[\rm(i)] For all $\ell \in \N_0$, \eqref{eq:gradientflow1} admits a unique solution $\tt_h^{i,\ell} \in \dts\big[\nn_h^{i,\ell}\big]$;
\item[\rm(ii)] The inner loop terminates in a finite number of iterations, i.e., there exists $\ell \in \N_0$
such that the stopping criterion~\eqref{eq:gradientflow2_stopping} is met;
\item[\rm(iii)] \eqref{eq:gradientflow3} admits a unique solution
$s_h^{i+1} \in V_h$ such that $s_h^{i+1}(z)=g_h(z)$ for all $z\in\nodes\cap\Gamma_D$.
\item[\rm(iv)] There holds
\begin{equation} \label{eq:energy_decrease}
\begin{split}
E^h[s_h^{i+1},\nn_h^{i+1}]
- E^h[s_h^i,\nn_h^i]
& \leq
- \left(
\tau_s\norm[L^2(\Omega)]{d_t s_h^{i+1}}^2
+
\tau_{\nn} \sum_{\ell=0}^{\ell_i}\norm[*]{\tt_h^{i,\ell}}^2
\right) \\
& \quad 
- \left( \tau_s^2 \, E_1^h[d_t s_h^{i+1},\nn_h^{i+1}] + \tau_{\nn}^2 \sum_{\ell=0}^{\ell_i}E_1^h[s_h^i,\tt_h^{i,\ell}] \right).
\end{split}
\end{equation}
In particular, $E^h[s_h^{i+1},\nn_h^{i+1}] \le E^h[s_h^i,\nn_h^i]$
and equality holds if and only if $(s_h^{i+1},\nn_h^{i+1}) = (s_h^i,\nn_h^i)$ (equilibrium state).
\end{itemize}
\end{proposition}

\begin{remark}[energy decrease]
The right-hand side of~\eqref{eq:energy_decrease}
characterizes the energy decrease
guaranteed by each step of Algorithm~\ref{alg:gradient_flow}
and comprises two contributions:
The term
\begin{equation*}
- \left(
\tau_s\norm[L^2(\Omega)]{d_t s_h^{i+1}}^2
+
\tau_{\nn} \sum_{\ell=0}^{\ell_i}\norm[*]{\tt_h^{i,\ell}}^2
\right)
\end{equation*}
is the energy decrease due to the gradient-flow nature
of Algorithm~\ref{alg:gradient_flow}.
The term
\begin{equation*}
- \left( \tau_s^2 \, E_1^h[d_t s_h^{i+1},\nn_h^{i+1}] + \tau_{\nn}^2 \sum_{\ell=0}^{\ell_i}E_1^h[s_h^i,\tt_h^{i,\ell}] \right)
\end{equation*}
is the numerical dissipation due to the backward Euler methods used
for the time discretization.
\end{remark}

In practical implementations of Algorithm~\ref{alg:gradient_flow},
the outer loop is terminated when
\begin{equation} \label{eq:gradientflow_stopping}
\big\lvert E^h[s_h^{i+1},\nn^{i+1}_h]-E^h[s_h^i,\nn_h^i] \big\rvert < \tol.
\end{equation}
Since the algorithm fulfills a monotone energy-decreasing property
(see Proposition~\ref{prop:scheme}(iv)),
the stopping criterion is met in a finite number of iterations.

The approximations $\nn^{i+1}_h$ of the director generated by Algorithm~\ref{alg:gradient_flow}
do not satisfy the unit-length constraint at the vertices of the mesh, as in
\cite{nwz2017,nochetto2018ericksen}.
However, the following proposition, proved in Section~\ref{sec:prop-scheme},
shows that violation of this constraint can be controlled by the time-step size $\tau_{\nn}$,
independently of the number of iterations.
Moreover, the uniform boundedness in $\LL^{\infty}(\Omega)$ of the sequence
can be guaranteed if the discretization parameters are chosen appropriately.

\begin{proposition}[properties of discrete director field] \label{prop:properties_n}
Let $j \ge 1$. The following holds.
\begin{enumerate}[\rm(i)]
\item
Suppose that the norm induced by the metric $\inner[*]{\cdot}{\cdot}$ used in~\eqref{eq:gradientflow1}
is an upper bound for the $L^2$-norm, i.e.,
there exists $C_*>0$ such that
\begin{equation} \label{eq:metric_norm}
\norm[\LL^2(\Omega)]{\pphi_h} \le C_* \norm[*]{\pphi_h}
\quad
\text{for all }
\pphi_h \in \VVV_{h,D}.
\end{equation}
Then, the approximations generated by Algorithm~\ref{alg:gradient_flow} satisfy
\begin{equation} \label{eq:L1error}
\norm[L^1(\Omega)]{I_h\big[\abs{\nn_h^j}^2-1\big]}
\le
C_1 \tau_{\nn} \, E^h[s_h^0,\nn_h^0],
\end{equation}
where $C_1>0$ depends only on $C_*$ and the shape-regularity of $\{ \mesh \}$.

\item
Suppose $\tau_{\nn}$ fulfills the following CFL-type condition
\begin{equation} \label{eq:cfl}
\begin{aligned}
\tau_{\nn} h_{\min}^{-d} \le C^*
&\quad
\text{if $\inner[*]{\cdot}{\cdot}$ is chosen as~\eqref{eq:metricL2}}, \\
\tau_{\nn} h_{\min}^{2-d-\alpha} |\log h_{\min}|^2 \le C^*
&\quad
\text{if $\inner[*]{\cdot}{\cdot}$ is chosen as~\eqref{eq:metricH1_weighted}},
\end{aligned}
\end{equation}
where $h_{\min}:=\min_{K\in\mesh}h_K$ and $C^*>0$ is arbitrary.
Then, the approximations generated by Algorithm~\ref{alg:gradient_flow} satisfy
\begin{equation} \label{eq:Linfty_bound}
\norm[\LL^\infty(\Omega)]{\nn_h^j}
\le
1 + C_2 E^h[s_h^0,\nn_h^0],
\end{equation}
where $C_2>0$ is proportional to $C^*>0$ in~\eqref{eq:cfl} with proportionality
constant depending on the shape-regularity of $\{\mesh\}$.
\end{enumerate}
\end{proposition}

To conclude this section, we discuss the structure of Algorithm~\ref{alg:gradient_flow}
with special emphasis on its nested structure and distinct roles of $\tau_{\nn}$ and $\tau_s$. Obviously, $\tau_{\nn}$ controls the violation of the unit-length constraint according to \eqref{eq:L1error}, but the roles of subiterations in \eqref{eq:gradientflow1} and $\tau_s$ in \eqref{eq:gradientflow3} is more subtle and deserves further elaboration. The presence of defects is associated with values $s_h^i(x_j)$ close to zero at nodes $x_j$, which in turn act as weights in the equation \eqref{eq:gradientflow1} for the tangential updates $\tt_h^{i,\ell}$ of the director field $\nn_h^{i,\ell}$. The fast decrease to zero of $s_h^i(x_j)$, relative to the growth of $\Grad\nn_h^i$ in its vicinity, impedes further changes of $\nn_h^i(x_j)$ because they are not energetically favorable: The defect is thus pinned at the same location $x_j$ for many interations. Experiments with Algorithm \ref{alg:gradient_flow} reveal  defect pinning if $\tau_{\nn}=\tau_s$ and one step of \eqref{eq:gradientflow1} per step of \eqref{eq:gradientflow3} is utilized. The subiterations within the inner loop \eqref{eq:gradientflow1} allow $\nn_h^{i,\ell}$ to adjust to the current value of $s_h^i$. This mimics an approximate optimization step but with unit length and max norm control dictated by Proposition \ref{prop:properties_n}. In contrast, full optimization has been proposed in \cite{nwz2017,nochetto2018ericksen,walker2020} instead of \eqref{eq:gradientflow1}, followed by nodal projection onto the unit sphere, whereas one step of a weighted gradient flow \eqref{eq:gradientflow1} has been advovated in \cite{BNW2020} for the $Q$-tensor model. On the other hand, since $\tau_s$ penalizes changes of $s_h^i$, smaller values of $\tau_s$ relative to $\tau_{\nn}$ delay changes of $s_h^i$ in favor of changes of $\nn_h^i$. This does not fix the stiff character of \eqref{eq:gradientflow1}, studied in \cite{carter2020domain}, but does remove defect pinning. Several numerical experiments in Section \ref{sec:numerics} document this finding.

\section{Numerical experiments} \label{sec:numerics}

In this section,
we present a series of numerical experiments that explore the accuracy
of Algorithm~\ref{alg:gradient_flow} and its ability to approximate
rather complex defects of nematic LCs in 2D and 3D. In both cases, these
results complement
the theory of Sections \ref{sec:discretization} and \ref{sec:algorithm}
and extend it.

We have implemented Algorithm~\ref{alg:gradient_flow}
within the high performance multiphysics
finite element software Netgen/NGSolve \cite{schoberl2017netgen}.
To solve the constrained variational problem~\eqref{eq:gradientflow1},
we adopt a saddle point approach.
The ensuing linear systems are solved using the built-in conjugate gradient solver of Netgen/NGSolve, while the visualization relies on ParaView~\cite{Ahrens2005}.

All pictures below obey the following rules. The vector field depicts the director $\nn$,
whereas the color scale refers to the degree of orientation $s$.
Blue regions indicate areas with values of $s$ close to zero, which signify the
occurrence of defects, while the red ones indicate regions with largest values
of $s$ ($s \approx 0.75$ in our simulations), where the director encodes the local
orientation of the LC molecules. We generate unstructured, generally
non-weakly acute, meshes within Netgen with desirable mesh size $h_0$ but the 
effective maximum size $h$ of tetrahedra in 3D may only satisfy $h\approx h_0$.
For the sake of reproducibility, we will specify $h_0$ when dealing with unstructured 3D meshes.

We stress that, unlike FEMs proposed in previous works~\cite{nochetto2018ericksen,nwz2017}, the energy-decreasing property of Algorithm~\ref{alg:gradient_flow}
does rely on meshes being weakly acute (cf.\ Proposition~\ref{prop:scheme}). Except for simple 3D geometries, such meshes are hard, to impossible, to construct. This is the case of the cylinder domain in Section~\ref{sec:cylinder_ex} and the Saturn ring configurations in Section~\ref{sec:saturn}, for which mesh flexibility is of fundamental importance to capture topologically complicated defects.

Throughout this section,
we consider the double well potential $\psi(s)=\cdw(\psi_c(s)-\psi_e(s))$ with
\begin{equation} \label{dw-func}
  \psi_c(s) := 63s^2,
  \quad
  \psi_e(s) := -16s^4+\frac{64}{3}s^3+57s^2-0.5625,
\end{equation}
where $\cdw \ge 0$.
Note that, for $\cdw > 0$,
$\psi$ has a local minimum at $s=0$ and a global minimum at $s=\hat{s}:=0.750025$ such that $\psi(\hat{s})=0$.
Moreover, in view of Proposition~\ref{prop:properties_n} 
we measure the violation of the unit-length constraint in terms of the quantity
\begin{equation} \label{eq:unit_length_error}
\mathrm{err}_{\nn} := \norm[L^1(\Omega)]{I_h\big[\abs{\nn_h^N}^2-1\big]},
\end{equation}
where $\nn_h^N$ denotes the final approximation of the director generated by Algorithm~\ref{alg:gradient_flow}.
Furthermore, unless otherwise specified,
we choose the $L^2$-metric~\eqref{eq:metricL2}
in~\eqref{eq:gradientflow1}, and
we set the tolerance $\tol=10^{-6}$ in both~\eqref{eq:gradientflow2_stopping}
and~\eqref{eq:gradientflow_stopping}.

\subsection{Point defect in 2D} \label{sec:point2D}

In striking contrast with the Oseen--Frank model, the Ericksen model allows
point defects to have finite energy in 2D:
The blow-up of $\abs{\Grad\nn}$ near a defect is compensated by infinitesimal
values of $s$ for the energy $E[s,\nn]$ in~\eqref{eq:Ericksen-energy} to stay bounded.
We examine this basic mechanism with simulations of a point defect in 2D
and study the influence of the discretization parameters
on the performance of Algorithm~\ref{alg:gradient_flow}.

We consider the unit square $\Omega = (0,1)^2$,
and set $\kappa=2$ in~\eqref{eq:Ericksen-energy}
as well as $\cdw = 0.1(0.3)^{-2}$ in~\eqref{dw-func}.
We impose Dirichlet boundary conditions for $s$ and $\nn$
on $\partial\Omega$, namely
\begin{equation}\label{BC-2D}
g=\hat{s} \quad \text{and} \quad \qq=\rr/g=\frac{(x-0.5,y-0.5)}{|(x-0.5,y-0.5)|}
\quad \text{on } \partial\Omega.
\end{equation}
To initialize Algorithm~\ref{alg:gradient_flow},
we consider a constant degree of orientation
$s_h^0 = \hat{s}$ in $\Omega$
and a director $\nn_h^0$ exhibiting an off-center point defect located at $(0.24,0.24)$.
Due to the imposed boundary conditions and for symmetry reasons,
we expect that an energy-decreasing dynamics moves the defect to the center of the square;
see Figure~\ref{fig:2D-point-defect}.

\begin{figure}[htbp]
	\begin{center}
		\includegraphics[width=4.9cm]{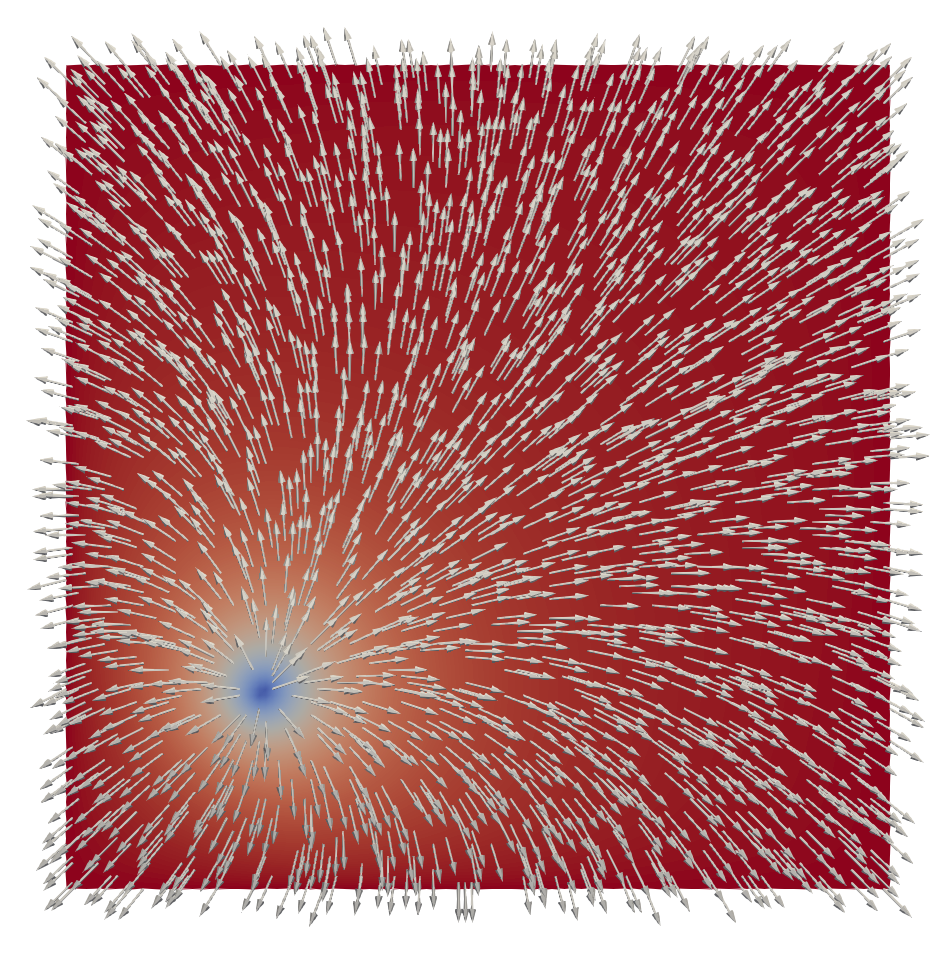}
		\includegraphics[width=5.0cm]{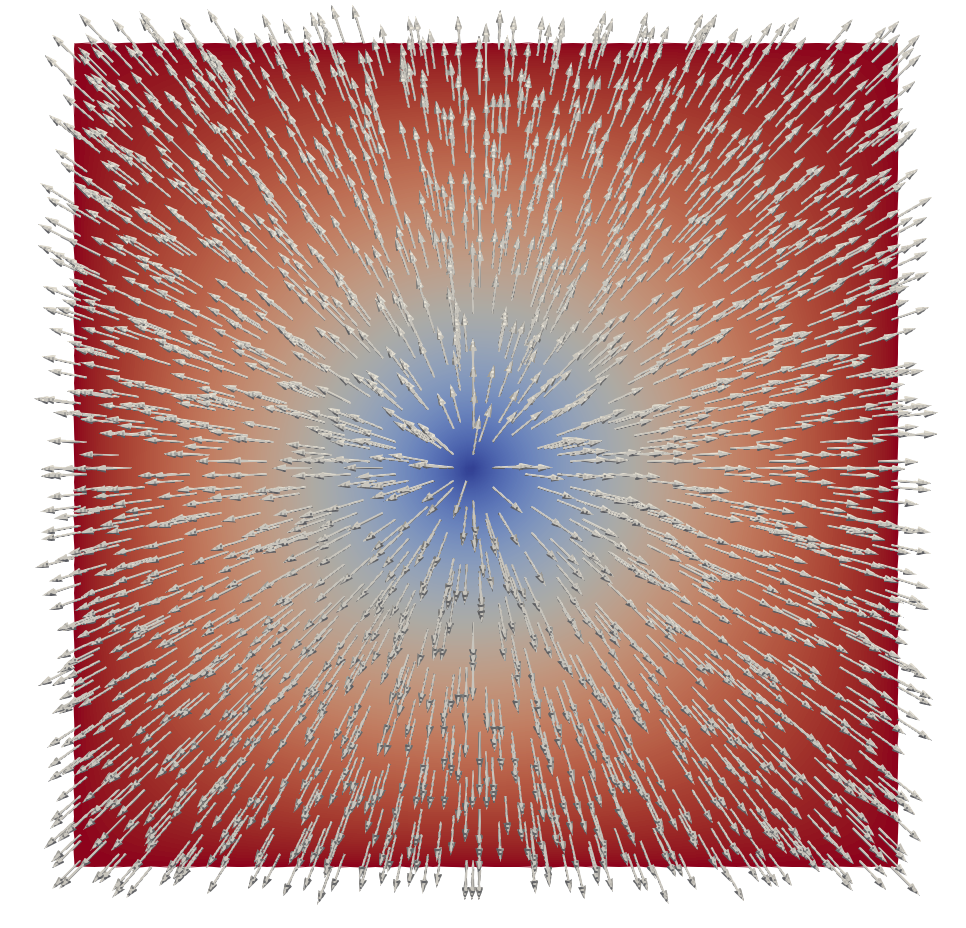}
		\caption{Point defect experiment of Section~\ref{sec:point2D}:
		Plot of the approximation $(s_h^1,\nn_h^1)$ after the first iteration (left)
		and of the final approximation $(s_h^N,\nn_h^N)$ (right). The gradient flow algorithm moves the defect to the center of the domain.}
		\label{fig:2D-point-defect}
	\end{center}
\end{figure}

In our first experiment, we consider a uniform mesh $\mesh$ of the unit square consisting of
2048 right triangles. The resulting mesh size is $h = \sqrt{2}\,2^{-5}$.
Moreover, we set $\tau_{\nn}=\tau_s=0.1$
and compare the results obtained for different choices of the metric $\inner[*]{\cdot}{\cdot}$ in~\eqref{eq:gradientflow1};
cf.\ \eqref{eq:metricL2}--\eqref{eq:metricH1_weighted}.
Table~\ref{table:2D-point-defect} displays the outputs for each run.
On the one hand, we observe that using the $L^2$-metric leads to the fastest dynamics
in terms of both number of iterations and CPU time.
On the other hand, the violation of the unit-length constraint is smaller for the weighted $H^1$-metrics.
For smaller values of $\alpha$ in the weighted $H^1$-metric,
Algorithm~\ref{alg:gradient_flow} terminates with a configuration exhibiting defect pinning at an off-center location.
The expected equilibrium state, depicted in Figure~\ref{fig:2D-point-defect} (right), can be restored when
reducing the time-step size $\tau_s$.
\begin{table}[htbp]
\begin{center}
\begin{tabular}{|c|c|c|c|c|c|}
\hline
metric					& $N$	& $E^h[s_h^N,\nn_h^N]$ & $\min(s_h^N)$ & $\mathrm{err}_{\nn}$ & CPU time (in s) \\ \hline
$L^2$					& 60		& 2.984	& 0.0757	& 0.0404	& 64.83 \\ \hline
weighted $H^1$, $\alpha=2.0$ 	& 67		& 2.944	& 0.0750	& 0.0370	& 98.65 \\ \hline
weighted $H^1$, $\alpha=1.9$	& 65 		& 2.938	& 0.0754	& 0.0362	& 111.69 \\ \hline
weighted $H^1$, $\alpha=1.8$	& 67 		& 2.932	& 0.0755	& 0.0353	& 130.17 \\ \hline
weighted $H^1$, $\alpha=1.7$ 	& 80		& 2.926	& 0.0760	& 0.0342	& 154.92 \\
\hline
\end{tabular}
\medskip
\caption{
Point defect experiment of Section~\ref{sec:point2D}:
Final outputs of Algorithm~\ref{alg:gradient_flow}
for different choices of metric $\inner[*]{\cdot}{\cdot}$, namely
total number of iterations $N$,
value of the energy $E^h[s_h^N,\nn_h^N]$ for the equilibrium state,
smallest value of the final $s_h^N$,
error in the unit-length constraint in~\eqref{eq:unit_length_error},
and the CPU time.}
\label{table:2D-point-defect} 
\end{center}
\end{table}

In our second set of experiments,
we investigate the effect of mesh refinement and changes of the time-step size on the results.
To this end, we first repeat the simulation using three uniform meshes with $h = \sqrt{2} \, 2^{-5-\ell}$ ($\ell = 0,1,2$); we set $\tau_{\nn} = 0.1 \,  2^{-2\ell}$,
in agreement with the CFL condition in~\eqref{eq:cfl} for the $L^2$-metric and $d=2$.
We collect the results of computations in Table \ref{table:2D-point-defect-2} (left),
and observe that both $\min(s_h^N)$ and $\mathrm{err}_{\nn}$ decrease about linearly
with $h$, whereas the energy $E_1^h[s_h^N,\nn_h^N]$ also decreases.
We next consider a fixed mesh with $h = \sqrt{2}\,2^{-5}$
and study the decay of $\mathrm{err}_{\nn}$ in~\eqref{eq:unit_length_error} as
the time-step size $\tau_{\nn}$ decreases;
see Table \ref{table:2D-point-defect-2} (right).
In this third set of experiments, we let
$\tau_{\nn} = (0.1)2^{-5-\ell}$ ($\ell = 0,1,2$), and $\tol=10^{-5}\tau_{\nn}$ in both~\eqref{eq:gradientflow2_stopping} and~\eqref{eq:gradientflow_stopping}.
The computational results in Table~\ref{table:2D-point-defect-2} (right) confirm
the first-order convergence with respect to $\tau_{\nn}$ established in Proposition~\ref{prop:properties_n}; see \eqref{eq:L1error} that bounds $\mathrm{err}_{\nn}$ in terms of
$\tau_{\nn} E^h[s_h^0,\nn_h^0]$. This explains the behavior of $\mathrm{err}_{\nn}$
in Table~\ref{table:2D-point-defect-2} (left) upon refinement, which increases
$E^h[s_h^0,\nn_h^0]$ because $\nn_h^0$ has a point defect while $s_h^0$ is constant and does not compensate the blow of $\Grad\nn_0$.

\begin{table}[htbp]
\begin{center}
\begin{tabular}{|c|c|c|c|c|c|}
\hline
$h$				& $N$	& $E^h[s_h^N,\nn_h^N]$ & $\min(s_h^N)$ & $\mathrm{err}_{\nn}$ & CPU time (in s)\\ \hline
$\sqrt{2} \, 2^{-5}$	& 60		& 2.984	& 0.0757	& 0.0404	& 64.83 \\ \hline
$\sqrt{2} \, 2^{-6}$	& 61		& 2.940	& 0.0422	& 0.0232	& 592.23 \\ \hline
$\sqrt{2} \, 2^{-7}$	& 133	& 2.939	& 0.0289	& 0.0100	& 7919.25\\
\hline
\end{tabular}
\hfil
\begin{tabular}{|c|c|}
\hline
$\tau_{\nn}$ & $\mathrm{err}_{\nn}$ \\ \hline
$(0.1)2^{-5}$	& 0.00610 \\ \hline
$(0.1)2^{-6}$	& 0.00346	 \\ \hline
$(0.1)2^{-7}$	& 0.001927 \\
\hline
\end{tabular}
\medskip
\caption{
Point defect experiment of Section~\ref{sec:point2D}:
Final outputs of Algorithm~\ref{alg:gradient_flow}
for different uniform meshes with mesh size $h$ and time steps $\tau_{\nn} = Ch^2$
(left) and different time step sizes $\tau_{\nn}$ with fixed mesh size
$h=\sqrt{2}\,2^{-5}$ (right).}
\label{table:2D-point-defect-2} 
\end{center}
\end{table}

\subsection{Plane defect in 3D} \label{sec:plane3D}

We simulate a plane defect in the unit cube $\Omega = (0,1)^3$ located at $\{z=0.5\}$,
according to~\cite[Section~6.4]{virga1994}.
We set $\kappa=0.2$ in~\eqref{eq:Ericksen-energy}
and $\cdw = 0$ in~\eqref{dw-func}.
We impose Dirichlet boundary conditions on the top and bottom faces $\Gamma_D$
of the cube
\begin{equation*}
g=\hat{s}, \
\qq=\rr/g=(1,0,0)
\text{ on } \partial\Omega \cap \{ z=0\},
\quad
g=\hat{s}, \
\qq=\rr/g=(0,1,0)
\text{ on } \partial\Omega \cap \{ z=1\}.
\end{equation*}
The exact solution is $\nn(z)=(1,0,0)$ for $z<0.5$ and $\nn(z)=(0,1,0)$ for $z>0.5$, while $s(z)=0$ on $z=0.5$ and linear on $(0,0.5)\cup(0.5,1)$~\cite[Section~6.4]{virga1994}.
Our numerical results are consistent with those in~\cite[Section~5.3]{nwz2017}.
To initialize Algorithm~\ref{alg:gradient_flow},
we set $s_h^0 =\hat{s}$ and $\nn_h^0$ to be a regularized point defect away from the center of the cube.
Figure \ref{fig:3D-plane-defect} displays the three components of $\nn_h^k$ and $s_h^k$ evaluated along the vertical line
$(0.5,0.5,z)$ for iterations $k=1, 31, 79$ computed on a uniform mesh with $h=\sqrt{3} \, 0.05$ and $\tau_{\nn}=\tau_s= 0.01$.

\begin{figure}[htbp]
\begin{tikzpicture}
\pgfplotstableread{data/iter1.dat}{\one}
\begin{axis}
[
width = 0.33\textwidth,
height=4.5cm,
title={$k=1$},
xlabel={$z$},
ylabel={director components},
xmin = 0,
xmax = 1,
ymin = -0.05,
ymax = 1.05,
legend style={at={(0.5,0.45)},anchor=north}
]
\addplot[blue,thick,densely dashed]	table[x=z, y=n1]{\one};
\addplot[red,thick,densely dotted]	table[x=z, y=n2]{\one};
\addplot[teal,thick,solid]			table[x=z, y=n3]{\one};
\legend{
{$n_1$},
{$n_2$},
{$n_3$},
}
\end{axis}
\end{tikzpicture}
\hfill
\begin{tikzpicture}
\pgfplotstableread{data/iter31.dat}{\one}
\begin{axis}
[
width = 0.33\textwidth,
height=4.5cm,
title={$k=31$},
xlabel={$z$},
xmin = 0,
xmax = 1,
ymin = -0.05,
ymax = 1.05,
legend style={at={(0.97,0.67)},anchor=east}
]
\addplot[blue,thick,densely dashed]	table[x=z, y=n1]{\one};
\addplot[red,thick,densely dotted]	table[x=z, y=n2]{\one};
\addplot[teal,thick,solid]			table[x=z, y=n3]{\one};
\legend{
{$n_1$},
{$n_2$},
{$n_3$},
}
\end{axis}
\end{tikzpicture}
\hfill
\begin{tikzpicture}
\pgfplotstableread{data/iter79.dat}{\one}
\begin{axis}
[
width = 0.33\textwidth,
height=4.5cm,
title={$k=79$},
xlabel={$z$},
xmin = 0,
xmax = 1,
ymin = -0.05,
ymax = 1.05,
legend style={at={(0.97,0.67)},anchor=east}
]
\addplot[blue,thick,densely dashed]	table[x=z, y=n1]{\one};
\addplot[red,thick,densely dotted]	table[x=z, y=n2]{\one};
\addplot[teal,thick,solid]			table[x=z, y=n3]{\one};
\legend{
{$n_1$},
{$n_2$},
{$n_3$},
}
\end{axis}
\end{tikzpicture}\\
\begin{tikzpicture}
\pgfplotstableread{data/iter1.dat}{\one}
\begin{axis}
[
width = 0.33\textwidth,
height=4.5cm,
xlabel={$z$},
ylabel={degree of orientation},
xmin = 0,
xmax = 1,
ymin =  0,
ymax = 0.76,
legend style={legend pos=south east}
]
\addplot[red,thick,densely dashed]	table[x=z, y=s]{\one};
\legend{
{$s$},
}
\end{axis}
\end{tikzpicture}
\hfill
\begin{tikzpicture}
\pgfplotstableread{data/iter31.dat}{\one}
\begin{axis}
[
width = 0.33\textwidth,
height=4.5cm,
xlabel={$z$},
xmin = 0,
xmax = 1,
ymin =  0,
ymax = 0.76,
legend style={legend pos=south east}
]
\addplot[red,thick,densely dashed]	table[x=z, y=s]{\one};
\legend{
{$s$},
}
\end{axis}
\end{tikzpicture}
\hfill
\begin{tikzpicture}
\pgfplotstableread{data/iter79.dat}{\one}
\begin{axis}
[
width = 0.33\textwidth,
height=4.5cm,
xlabel={$z$},
xmin = 0,
xmax = 1,
ymin =  0,
ymax = 0.76,
legend style={legend pos=south east}
]
\addplot[red,thick,densely dashed]	table[x=z, y=s]{\one};
\legend{
{$s$},
}
\end{axis}
\end{tikzpicture}
\caption{
Plane defect of Section~\ref{sec:plane3D}:
Plots of the three components of $\nn_h^k$ (first row)
and plots of $s_h^k$ (second row) for iterations $k=1, 31, 79$.
In the final configuration ($k=N=79$),
the energy is $E^h[s_h^N,\nn_h^N] =0.247$, $\min(s_h^N)=0.0101$,
and $\mathrm{err}_{\nn}=0.0556$. Moreover, there is
a transition layer between about $z=0.4$ and $z=0.6$, and $s_h$
is almost linear in $(0,0.4)$ and $(0.6,1)$.}
\label{fig:3D-plane-defect}
\end{figure}
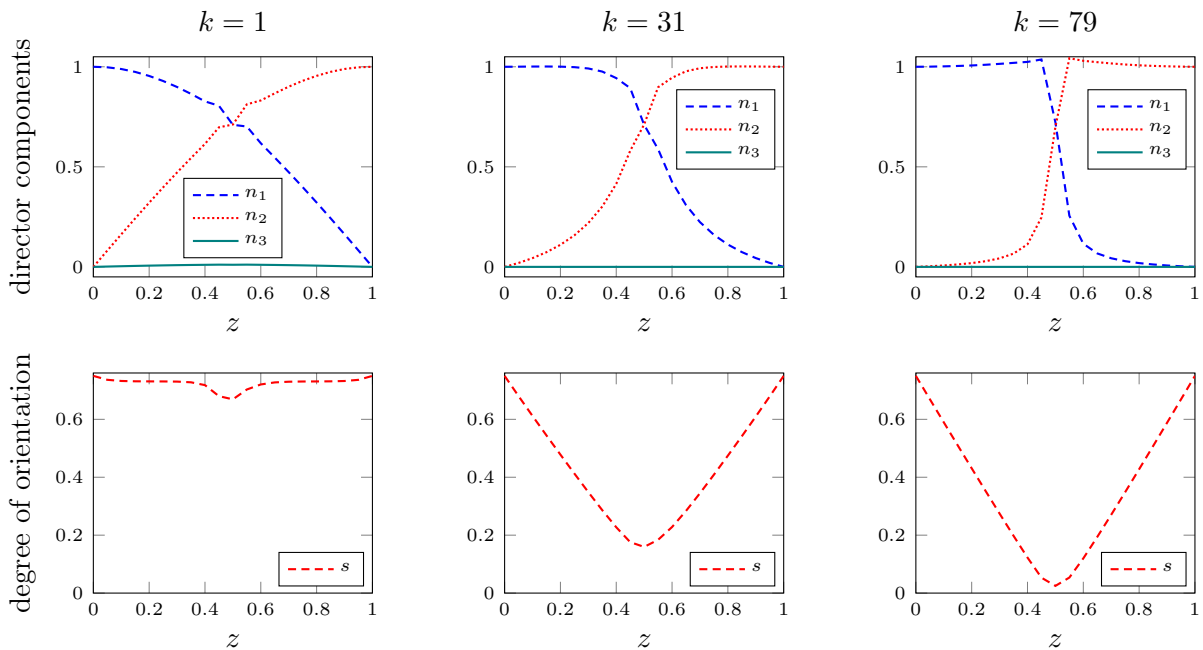

\subsection{Effect of $\kappa$ on equilibria} \label{sec:cylinder_ex}

The value of the constant $\kappa>0$ in~\eqref{eq:Ericksen-energy}
plays a crucial role in the formation of defects.
For large values of $\kappa$,
the dominant term in $E_1[s,\nn]$ is $\int_{\Omega} \kappa \abs{\grad s}^2$
that prevents variations of $s$. Typically $s$ tends to be close to a
(usually positive) constant
and the model behaves much like the simpler Oseen--Frank model,
where defects are less likely to occur
(and no defects with finite energy beyond point defects are allowed in 3D).
On the other hand,
for small values of $\kappa$,
the energy is dominated by $\int_{\Omega} s^2 \abs{\Grad \nn}^2$,
which allows $s$ to become zero to compensate large gradients of $\nn$,
and defects are then more likely to occur.
In this section, we investigate this dichotomy numerically.

We consider a cylindrical domain $\Omega$ in 3D with lateral
boundary $\Gamma_D$
\begin{gather*}
  \Omega = \{ (x,y,z) \in \R^3 : (x-0.5)^2+(y-0.5)^2 < 0.5^2, \, 0 < z < 1 \},
  \\
  \Gamma_D = \{ (x,y,z) \in \R^3 : (x-0.5)^2+(y-0.5)^2 = 0.5^2, \, 0 < z < 1 \},
\end{gather*}
and impose the Dirichlet conditions on $\Gamma_D$
\begin{equation}\label{BC-3D}
g=\hat{s}
\quad
\text{and}
\quad
\qq = \rr / g = \frac{(x-0.5,y-0.5,0)}{\abs{(x-0.5,y-0.5,0)}},
\end{equation}
The top and bottom faces of $\Omega$ are treated as free boundaries and the
double well potential $\psi$
is neglected, i.e., $\cdw=0$ in~\eqref{dw-func}.
The analysis in~\cite[Section~6.5]{virga1994} predicts
that minimizers of the energy exhibit
a line defect along the central axis of the cylinder
if $\kappa$ is sufficiently small,
whereas they are smooth (no defects) if $\kappa$ is sufficiently large.

Figure \ref{fig:cyl_2} displays the final configurations obtained for $\kappa=0.2$ and $\kappa=2$.
To discretize $\Omega$,
we consider an unstructured mesh generated by Netgen with $h_0=0.05$.
For both values of $\kappa$, we set $\hat{s}$ as initial condition for the degree of orientation.
For $\kappa=0.2$,
we set $\tau_{\nn}=0.1$ and $\tau_s=10^{-3}$ and take as initial condition for the director field
an off-center point defect located at the slice $z=0.5$. 
For $\kappa=2$, we set $\tau_{\nn}=\tau_s=0.01$
and initialize $\nn_h^0$ as an off-center point defect located at the slice $z=0.25$.
These computational results are consistent with those in~\cite{nwz2017}
and confirm the predicted effect of $\kappa$~\cite[Section~6.5]{virga1994}.

\begin{figure}[htbp]
\begin{center}
\includegraphics[width=6cm]{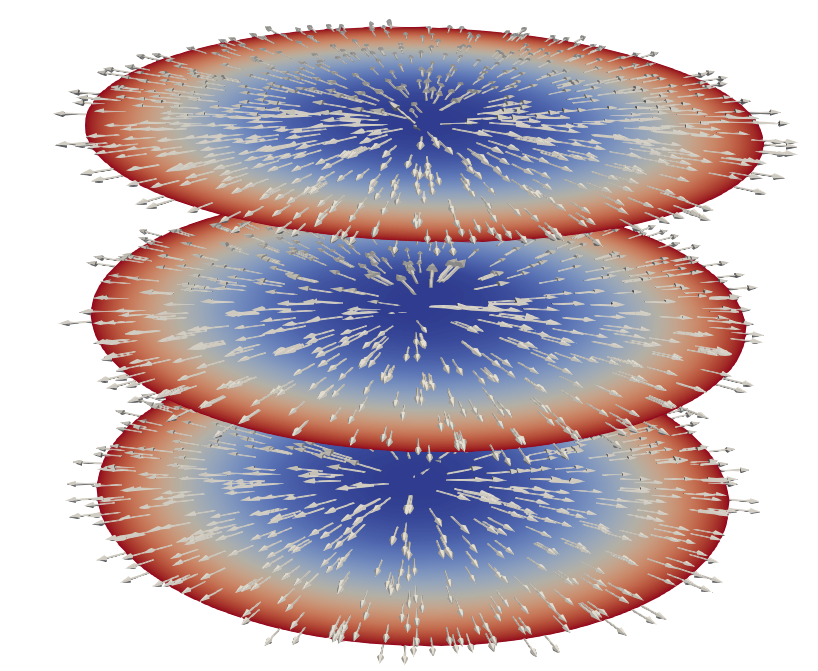}
\quad
\includegraphics[width=5.6cm]{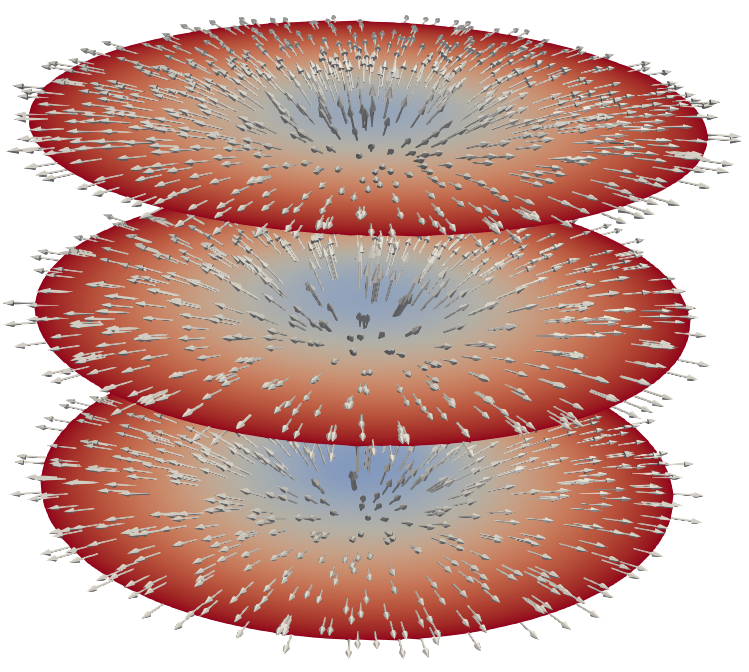}
\caption{  
Effect of $\kappa$ in Section~\ref{sec:cylinder_ex}:
Equilibria for $\kappa=0.2$ (left) and $\kappa=2$ (right).
Both pictures show $s_h^N$ and $\nn_h^N$ on the slices $z = 0.2, 0.5, 0.8$.
If $\kappa=0.2$, the final configuration exhibits a line defect along the central axis of the cylinder;
the final energy is $E^h[s_h^N,\nn_h^N] =0.806$,
$\min(s_h^N)=-7.33\times10^{-4}$,
$\mathrm{err}_{\nn}=0.0778$,
and $N = 226$.
If $\kappa=2$,
the $z$-component of the director is not zero.
This behavior is usually referred to as \emph{fluting effect}
or \emph{escape to the third dimension}~\cite[Section~6.5.1]{virga1994}.
Moreover, the degree of orientation is bounded well away from zero; the final energy is 
$E^h[s_h^N,\nn_h^N] =2.635$,
$\min(s_h^N)=0.224$,
$\mathrm{err}_{\nn}=0.044$,
and $N=17$.}
\label{fig:cyl_2}
\end{center}
\end{figure}

\subsection{Propeller defect}  \label{sec:propeller}

In this section, we investigate a new defect discovered
in~\cite[Section~5.4]{nwz2017}.
We consider a setup similar to the one discussed in Section~\ref{sec:cylinder_ex},
except that the domain is the unit cube $\Omega = (0,1)^3$, and we again
set $\cdw=0$ in~\eqref{dw-func}.
The top and bottom faces of the cube are treated as free boundary,
while the same strong anchoring conditions as in~\eqref{BC-3D}
are imposed on the vertical faces $\Gamma_D$ of the cube (lateral boundary).
The initial conditions
are $s_h^0 = \hat{s}$ for the degree of orientation
and an off-center point defect located on the slice $z=0.5$
for the director.
The domain is discretized using an unstructured mesh generated by Netgen with $h_0=0.025$, and we set $\tau_{\nn}=0.02$.
We consider the values $\kappa=2$ and $\kappa=0.1$.
For $\kappa=2$ and $\tau_s=0.2$, the computational results agree with those
of Section~\ref{sec:cylinder_ex}:
The equilibrium state is smooth and is characterized by a nonzero $z$-component
(fluting effect).

For $\kappa=0.1$, the final configuration reported in~\cite[Section~5.4, Figure~5]{nwz2017} consists of two plane defects intersecting at the vertical symmetry axis of the cube,
the so-called propeller defect. 
Whether this was a numerical artifact due to the inherent symmetries of the structured
uniform weakly acute meshes used in~\cite{nwz2017} for simulation was an intriguing open
question that we now answer.
Owing to the flexibility of our approach regarding meshes,
we repeated the experiment using an unstructured nonsymmetric mesh
with $\tau_s=10^{-4}$. Our computational results confirm the emergence of
the propeller defect in Figure~\ref{fig:k0p1}, which in turn
displays the director field $\nn_h^k$ at iterations $k=0, 1, 2766$
with colors indicating the size of $s_h^k$.
 
\begin{figure}[htbp]
\begin{center}
\includegraphics[width=4.8cm]{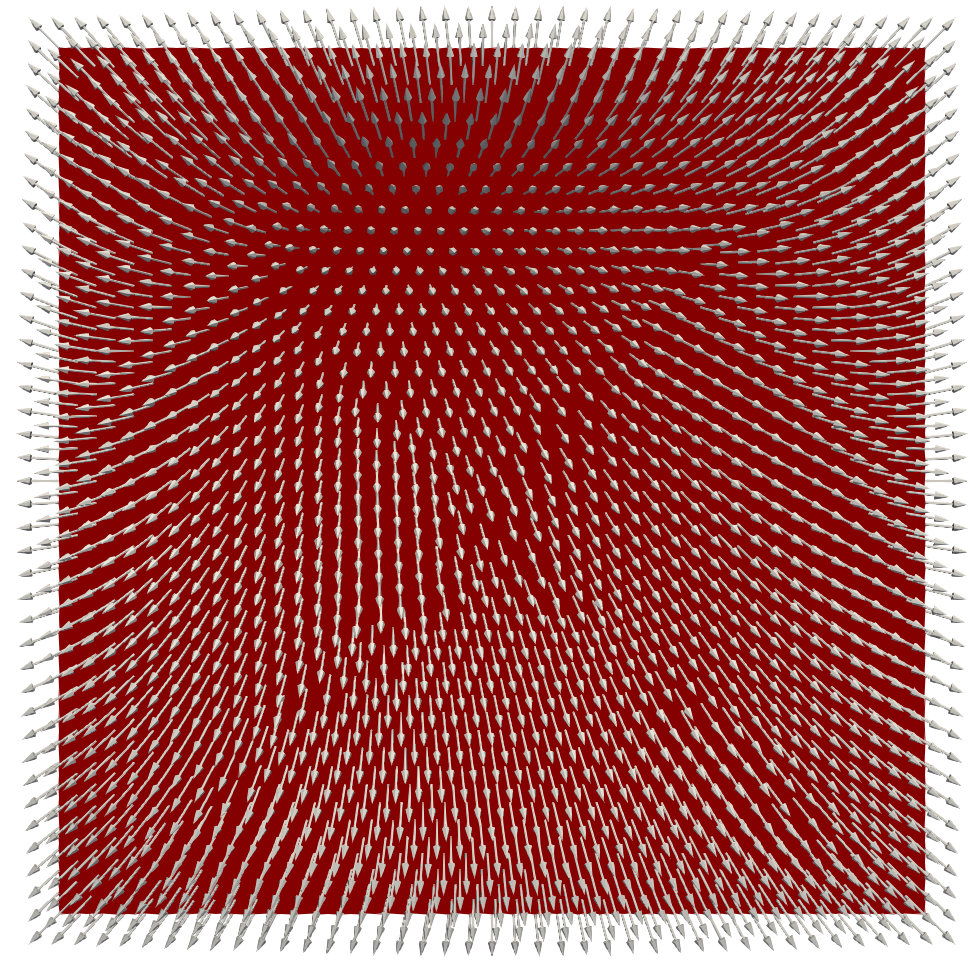}
\includegraphics[width=5cm]{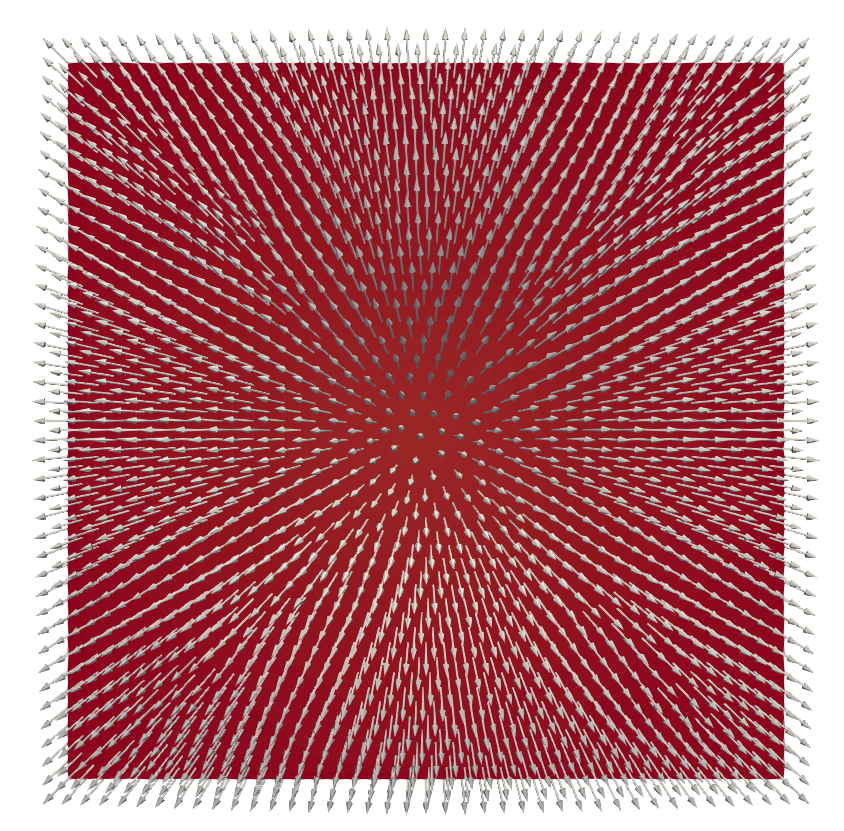}
\includegraphics[width=4.8cm]{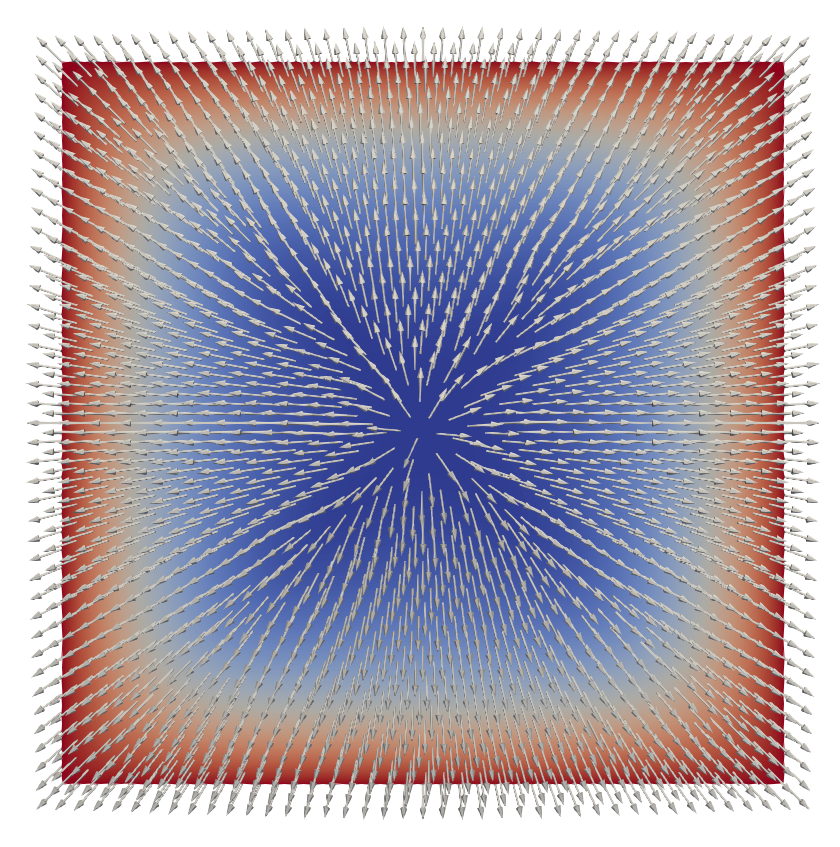}
\caption{
Propeller defect of Section~\ref{sec:propeller}:
Evolution of the order parameters on the top face of the cube ($z=1$).
Plots of the initial state $(s^0_h,\nn^0_h)$ (left),
of the intermediate approximation $(s^1_h,\nn^1_h)$ obtained after the first iteration (middle),
and of the equilibrium state $(s^N_h,\nn^N_h)$ after $2766$ iterations (right).
In the initial state,
due to the off-center point defect at $z=0.5$,
there is a corresponding region on the slice for $z=1$ where $\nn$ is aligned with $z$-direction.
After the first iteration,
in which $\nn$ is minimized for fixed $s = \hat{s}$,
by symmetry the defect has moved to the center on $z=0.5$.
Correspondingly,
on the top surface of the cube,
the region where $\nn$ is aligned with the $z$-axis
has moved to the center.
The final state is a propeller defect consisting of a planar X-like configuration extruded in the $z$-direction. The final energy is 
$E^h[s_h^N,\nn_h^N] =0.592$,
$\min(s_h^N)=-1.575\times10^{-4}$,
$\mathrm{err}_{\nn}=0.0265$,
and $N=2766$.}
\label{fig:k0p1}
\end{center}
\end{figure}

\subsection{Colloidal effects in nematic LCs} \label{sec:saturn}

Colloidal particles suspended in a nematic LC can induce
interesting topological defects and distortions~\cite{gu2000observation,stark1999director}.
One prominent example is the so-called \emph{Saturn ring defect},
a director configuration characterized by a circular ring singularity
surrounding a spherical particle and located around its equator. Such defects
are typically nonorientable and captured within the Landau--de Gennes $Q$-tensor
model~\cite{BNW2020,bw2021}, but the Ericksen model yields similar orientable defects
under suitable boundary conditions~\cite{nochetto2018ericksen}.
We confirm the ability of Algorithm~\ref{alg:gradient_flow} to produce
similar configurations.  

In this section, we exploit the flexibility of Algorithm~\ref{alg:gradient_flow}
regarding meshes, together with  the built-in Constructive Solid Geometry (CSG) approach of Netgen/NGSolve, to explore numerically the formation of Saturn-ring-like defects induced by 
nonspherical or multiple particles.

\subsubsection{One ellipsoidal particle}\label{sec:ellipsoid}

Let $\Omega_c = (0,1)^3$ be the unit cube
and let $\Omega_s \subset \Omega_c$ be an ellipsoid centered at $(0.5,0.5,0.5)$
with axes parallel to the coordinate axes and semiaxis lengths equal to $0.3$ ($x$-direction),
$0.075$ ($y$-direction), and $0.075$ ($z$-direction); $\Omega_s$ has an aspect
ratio $1:4$.
The computational domain is then $\Omega:=\Omega_c\setminus\overline{\Omega_s}$.
We set $\kappa=1$ in~\eqref{eq:Ericksen-energy}
as well as $\cdw = 0.2$ in~\eqref{dw-func}.
On $\partial\Omega = \partial\Omega_c \cup \partial\Omega_s$,
we impose strong anchoring conditions 
\begin{equation}\label{BC-saturn}
g=\hat{s}\text{ on }\partial\Omega,
\quad 
\qq=\rr/g=\nnu \text{ on } \partial\Omega_s,
\quad
\text{and}
\quad
\qq=\rr/g=\nn_{sr} \text{ on }\partial\Omega_c,
\end{equation} 
where $\nnu: \partial\Omega_s \to \mathbb{S}^2$ denotes the outward-pointing unit normal vector of $\Omega_s$
and $\nn_{sr} : \partial\Omega_s \to \mathbb{S}^2$ smoothly interpolates between
the constant values $(0, 0, -1)$ on the bottom face and $(0, 0, 1)$ on the top face of the cube
(see \cite[Figure~11]{nochetto2018ericksen}).
These boundary conditions are essential in order to induce the defect. 
The initial conditions for Algorithm~\ref{alg:gradient_flow} are given by
\begin{equation} \label{eq:saturnInitial}
s^0_h=\hat{s}
\text{ in }
\Omega
\quad
\text{and}
\quad
\nn^0_h(z) =
\begin{cases}
(0,0,1) & z\in\Omega \text{ and } z_3 \ge 0.5 , \\
(0,0,-1) & z\in\Omega \text{ and } z_3 < 0.5, \\
\qq(z) & z\in\partial\Omega,
\end{cases}
\end{equation}
for $z=(z_1,z_2,z_3)\in\nodes$.
Figure~\ref{fig:final_ellipsoid} displays cuts of the final configuration
obtained using Algorithm~\ref{alg:gradient_flow} with an unstructured mesh
with $h_0=0.05$ and time-step sizes $\tau_{\nn}=\tau_s=0.01$.

\begin{figure}[htbp]
\begin{center}
\includegraphics[width=4cm]{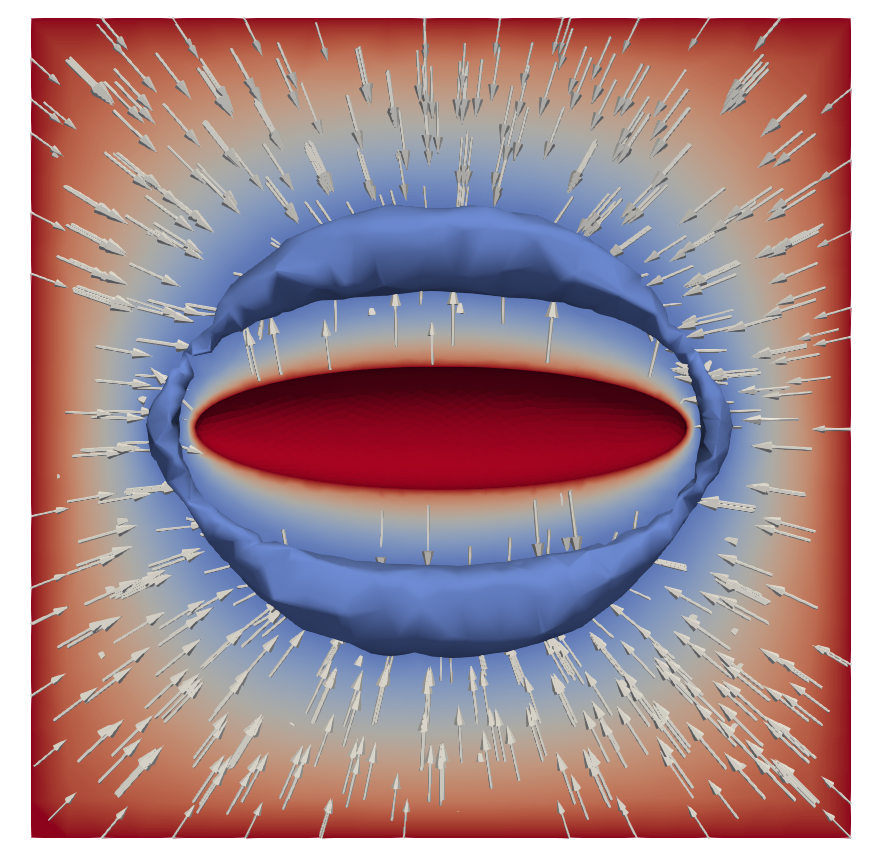}
\includegraphics[width=4cm]{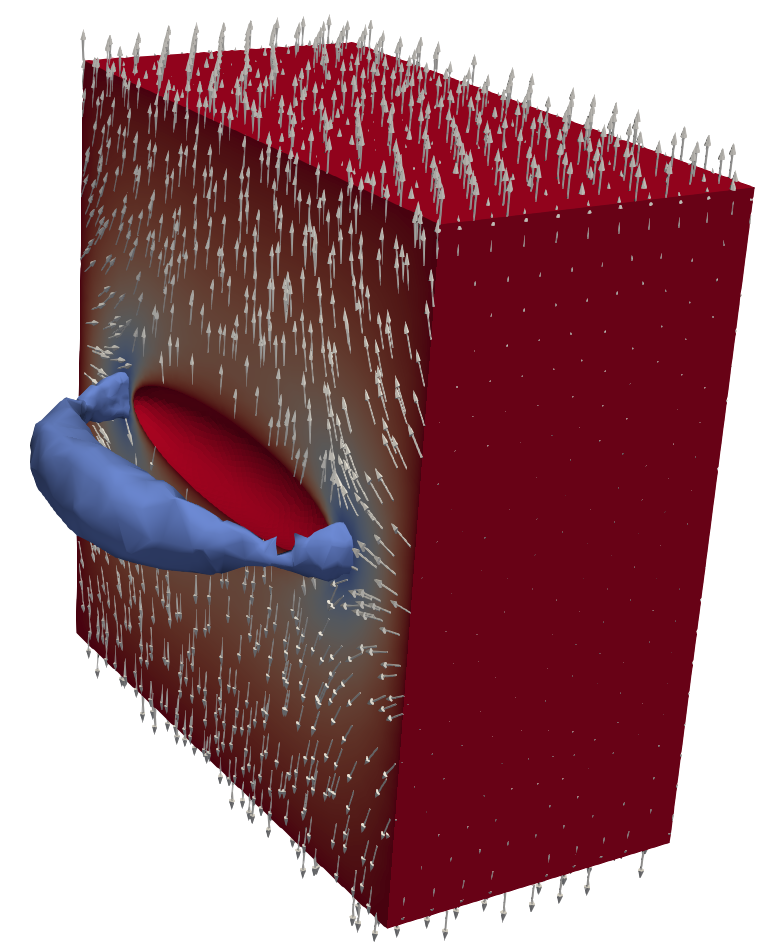}
\includegraphics[width=4cm]{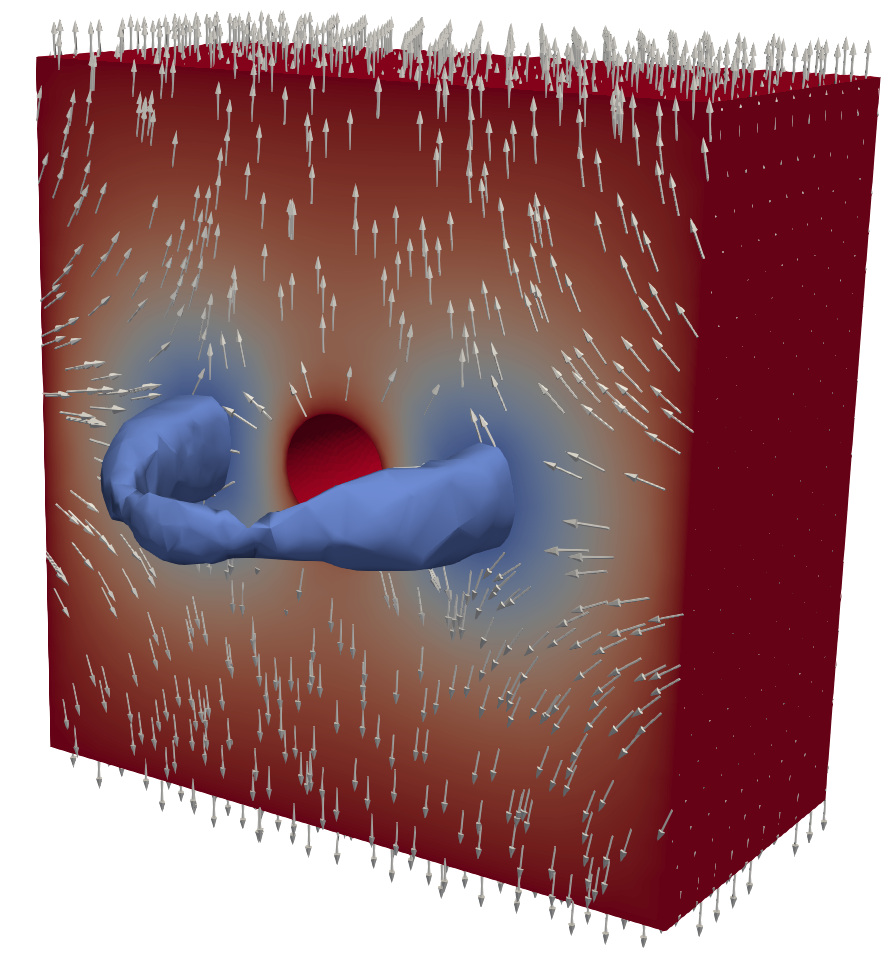}
\caption{Saturn ring experiment of Section~\ref{sec:ellipsoid}.
Three different perspectives of the Saturn ring defect around an ellipsoidal particle:
slice $z=0.5$ (left),
a 3D view clipped at $y=0.5$ (middle),
and a 3D view clipped at $x=0.5$ (right).
The blue ring surrounding the particle, the iso-surface for $s=0.15$, provides a good approximation of the defect.
We stress that neither the distance between the defect and the particle nor the defect diameter are constant, which is a consequence of the anisotropic shape of the particle.
The final energy is 
$E^h[s_h^N,\nn_h^N] =7.263$,
$\min(s_h^N)= 0.0128$,
$\mathrm{err}_{\nn}=0.145$,
and $N=33$.}
\label{fig:final_ellipsoid}
\end{center}
\end{figure}

\subsubsection{Multiple spherical particles} \label{sec:saturn_more}

We conclude this section with two novel and challenging
simulations involving multiple spherical colloidal particles.
In both cases, the domain has the form
$\Omega:=\Omega_c\setminus\overline{\Omega_s}$,
where
$\Omega_c \subset \R^3$
denotes a simply connected domain
(representing the LC container),
whereas
$\Omega_s \subset \Omega_c$ denotes the region occupied by
spherical colloidal particles.
We set $\kappa=1$ in~\eqref{eq:Ericksen-energy}
and $\cdw = 0.2$ in~\eqref{dw-func}.
Moreover, boundary and initial conditions are suitable extensions
to the multiple particle case of~\eqref{BC-saturn} and~\eqref{eq:saturnInitial}
considered in Section~\ref{sec:ellipsoid}.

Figure~\ref{fig:final_saturn_two} shows the equilibrium state corresponding
to $\Omega_c = (0,1)^3$ and a pair of disjoint spherical colloids $\Omega_s$
with radii $0.1$ and centered at $(0.3,0.5,0.5)$ and $(0.7,0.5,0.5)$.
Algorithm~\ref{alg:gradient_flow} employs
an unstructured mesh with $h_0=0.025$ and time-step sizes $\tau_{\nn}=\tau_s=0.0025$.
A novel \emph{fat figure ``8''} defect forms.

Figure~\ref{fig:final_saturn_six} depicts the equilibrium state corresponding
to $\Omega_c = (-0.1,1.1)^3$ and a colloidal region consisting of six spheres.
The latter have radii $0.1$ and centers located
at $(0.2,0.5,0.5)$, $(0.8,0.5,0.5)$, $(0.5,0.2,0.5)$, $(0.5,0.8,0.5)$, $(0.5,0.5,0.2)$, and $(0.5,0.5,0.8)$ distributed symmetrically with respect to the cube center.
Algorithm~\ref{alg:gradient_flow} utilizes an unstructured mesh with $h_0=0.05$
and time-step sizes $\tau_{\nn}=\tau_s=0.005$.

\begin{figure}[htbp]
\begin{center}
\includegraphics[width=3.7cm]{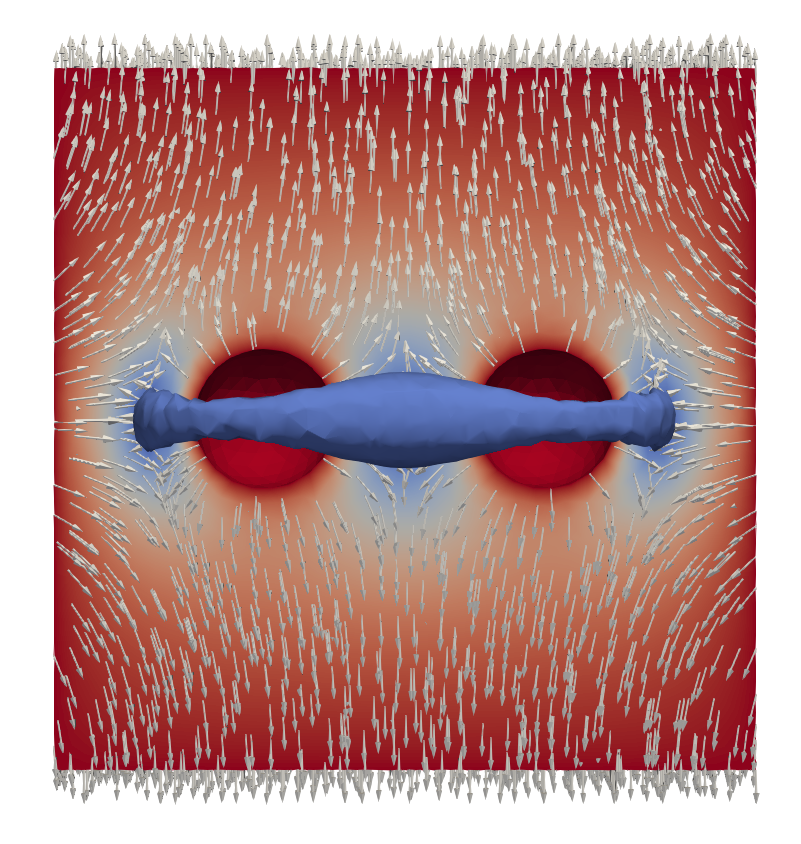}
\includegraphics[width=3.7cm]{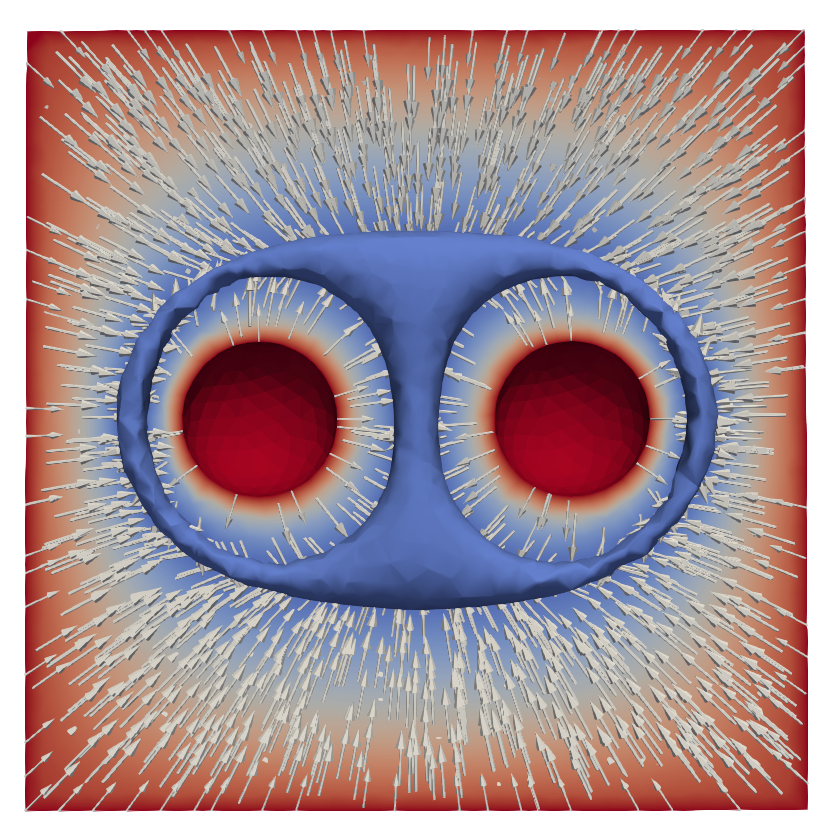}
\includegraphics[width=3.7cm]{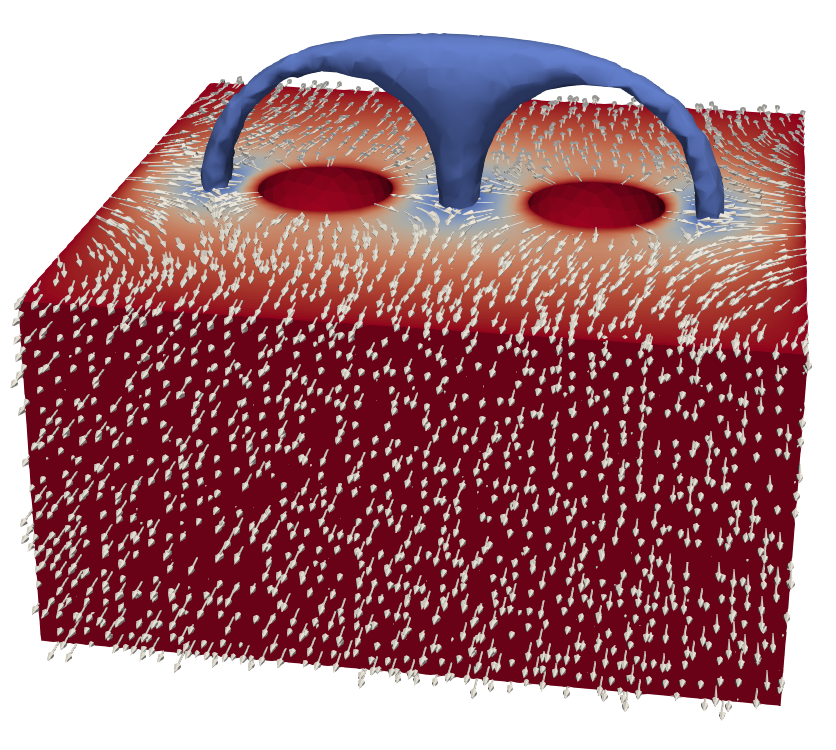}
\caption{Two-particle experiment of Section~\ref{sec:saturn_more}.
\emph{Fat figure ``8'' defect} around two spherical colloids viewed from different perspectives:
slice $y=0.5$ (left),
slice $z=0.5$ (middle),
and a 3D view clipped at $y=0.5$ (right).
The blue ring surrounding the particle is the iso-surface for $s=0.12$, which provides a good approximation of the defect. The final energy is 
$E^h[s_h^N,\nn_h^N] =7.656$,
$\min(s_h^N)= 0.0146$,
$\mathrm{err}_{\nn}=0.0972$,
and $N=57$.}
\label{fig:final_saturn_two}
\end{center}
\end{figure}

\begin{figure}[htbp]
\begin{center}
\includegraphics[width=4cm]{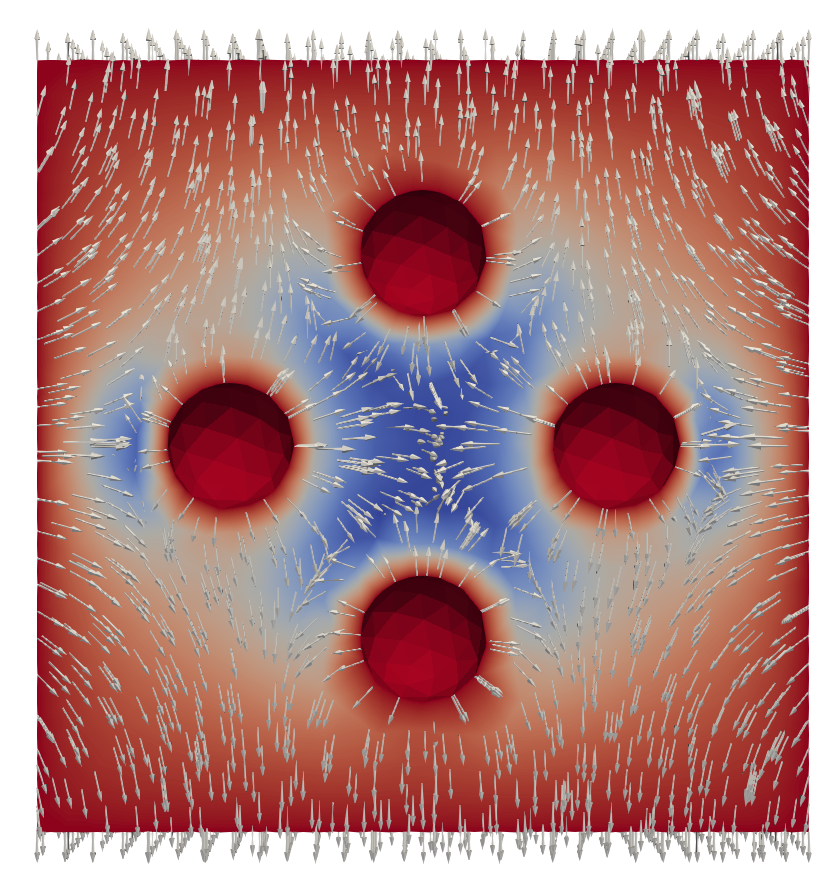} 
\includegraphics[width=4cm]{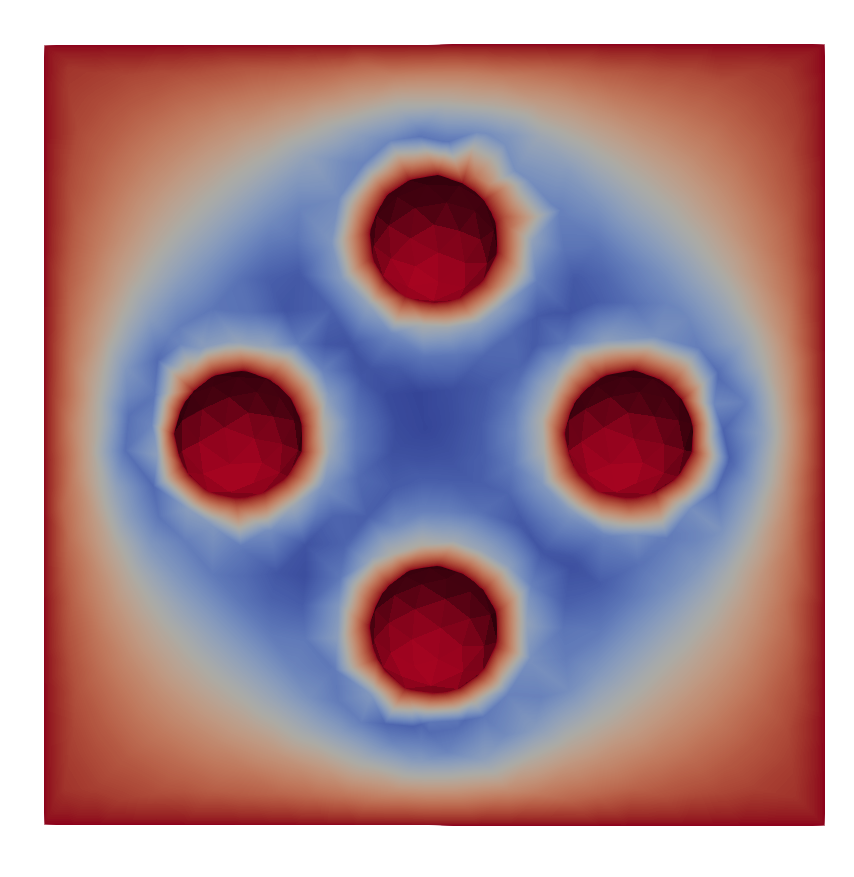}
\includegraphics[width=4cm]{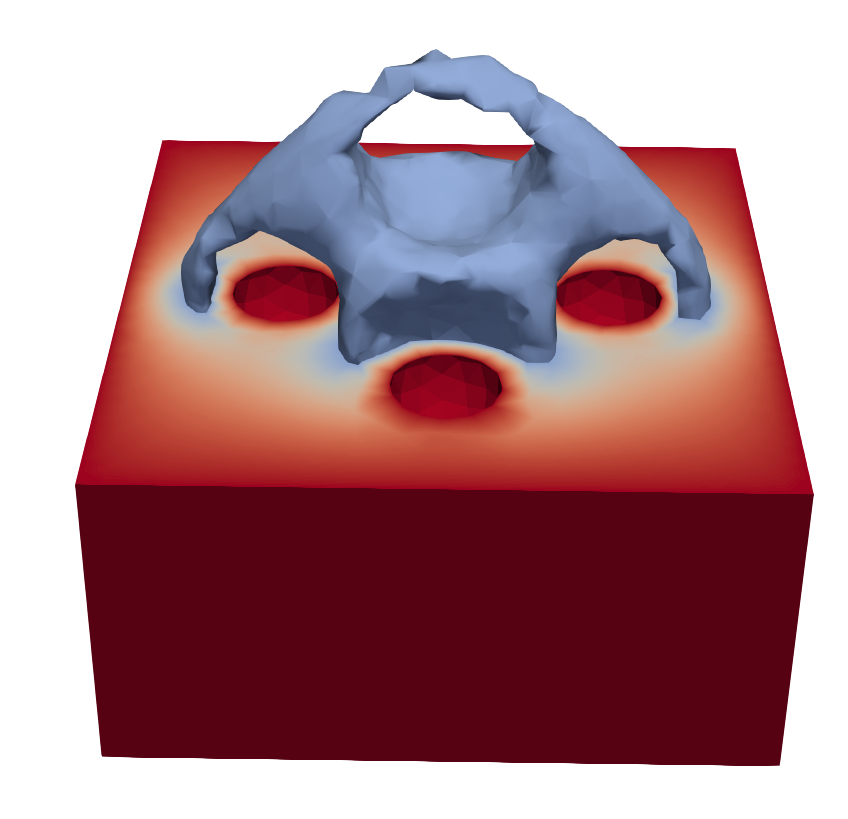}
\caption{Six-particle experiment of Section~\ref{sec:saturn_more}.
Defect around six spherical colloids viewed from different perspectives:
slice $y=0.5$ (left),
slice $z=0.5$ (middle),
and a 3D view clipped at $y=0.5$ (right);
the slice $x=0.5$ is similar to $y=0.5$. 
The blue ring surrounding the particles is the iso-surface for $s=0.22$,
which provides a good approximation of the defect.
Therefore the defect appears to be a combination of a large Saturn ring defect
around particles with center in the plane
$z=0.5$ and a planar X-like configuration with axis $x=0.5, y=0.5,-0.1<z<1$.
The final energy is 
$E^h[s_h^N,\nn_h^N] =12.562$,
$\min(s_h^N)= -0.0079$,
$\mathrm{err}_{\nn}=0.163$,
and $N=61$.}
\label{fig:final_saturn_six}
\end{center}
\end{figure}
  
\section{Proofs} \label{sec:proofs}

In this section, we present the proofs of the results discussed in 
Sections~\ref{sec:model}--\ref{sec:algorithm}.

\subsection{$L^2$-differentiability of admissible directors} \label{sec:L2diff}

We now prove that any admissible director field,
despite not being in $\HH^1(\Omega)$,
is $L^2$-differentiable in $\Omega\setminus\Sigma$. We refer to~\cite{BNW2020}
for a similar argument for a line field.

\begin{proof}[\bf Proof of Proposition~\ref{prop:L2diff}]
Since $(s,\nn,\uu) \in \A$, we have that $s \in H^1(\Omega)$ and $\uu = s \nn \in \HH^1(\Omega)$.
Then, for almost all $x \in \Omega$ (specifically, for all Lebesgue points of $(s,\uu,\grad s,\Grad \uu)$),
$s$ and $\uu$ are $L^2$-differentiable and their $L^2$-gradients coincide with their respective weak gradients for a.e. $x\in\Omega$,
i.e., as $r \to 0$, it holds that
\begin{align*}
\fint_{B_r(x)} \abs{s(y) - s(x) - \grad s(x) \cdot (y-x)}^2 \, \mathrm{d}y & = o(r^2), \\
\fint_{B_r(x)} \abs{\uu(y) - \uu(x) - \Grad \uu(x) (y-x)}^2 \, \mathrm{d}y & = o(r^2);
\end{align*}
see~\cite[Theorem~6.2]{eg2015}.
For almost all $x \in \Omega \setminus \Sigma$ 
(specifically, for all Lebesgue points of $(s,\nn,\uu,\grad s,\Grad \uu)$ in $x \in \Omega \setminus \Sigma$),
in view of the identity~\eqref{eq:productRule}, we define the quantity
\begin{equation} \label{eq:L2gradient}
\Grad \nn (x) :=
\frac{\Grad \uu(x) - \nn(x) \otimes \grad s(x)}{s(x)}.
\end{equation}
Let $r > 0$. It holds that
\begin{equation*}
\begin{split}
& \fint_{B_r(x)} \abs{\nn(y) - \nn(x) - \Grad \nn(x) (y-x)}^2 \, \mathrm{d}y \\
& \quad \lesssim \frac{1}{s(x)^2}\fint_{B_r(x)} \abs{\uu(y) - \uu(x) - \Grad \uu(x) (y-x)}^2 \mathrm{d}y \\
& \qquad + \frac{1}{s(x)^2}\fint_{B_r(x)} \abs{s(y) - s(x) - \grad s(x) \cdot (y-x)}^2 \abs{\nn(y)}^2 \mathrm{d}y \\
& \qquad + \frac{\abs{\grad s(x)}^2 }{s(x)^2}\fint_{B_r(x)} \abs{\nn(y) - \nn(x)}^2\abs{y-x}^2 \mathrm{d}y
= o(r^2)
\end{split}
\end{equation*}
as $r \to 0$.
This shows that $\Grad \nn (x)$ is the $L^2$-gradient of $\nn$ at $x$.
Moreover, \eqref{eq:structural2} follows from a direct computation.
In fact, in view of \eqref{eq:L2gradient}, there holds that
\begin{equation*}
\begin{split}
s(x)^2 \abs{\Grad\nn(x)}^2
& =
\abs{\Grad \uu(x) - \nn(x) \otimes \grad s(x)}^2 \\
& \stackrel{\phantom{\eqref{eq:L2gradient}}}{=} \abs{\Grad \uu(x)}^2
+ \abs{\nn(x) \otimes \grad s(x)}^2
- 2 \Grad\uu(x) : [\nn(x) \otimes \grad s(x)] \\
& \stackrel{\phantom{\eqref{eq:L2gradient}}}{=} \abs{\Grad \uu(x)}^2
- \abs{\grad s(x)}^2,
\end{split}
\end{equation*}
where the last equality follows from the identities
\begin{equation*}
\begin{split}
\abs{\nn(x) \otimes \grad s(x)}^2
= \sum_{i,j = 1}^d n_i(x)^2 \, \big(\partial_j s(x)\big)^2
= \sum_{j = 1}^d \big(\partial_j s(x)\big)^2
= \abs{\grad s(x)}^2
\end{split}
\end{equation*}
and for a.e $x \in \Omega \setminus \Sigma$
\begin{equation*}
\begin{split}
\Grad \uu (x) : [\nn(x) \otimes \grad s (x)]
& = \sum_{i,j = 1}^d \partial_j u_i(x) \, n_i (x) \, \partial_j s(x)
= \frac{1}{s(x)} \sum_{i,j = 1}^d \partial_j u_i(x) \, u_i (x) \, \partial_j s(x) \\
& = \frac{1}{2 s(x)} \sum_{i,j = 1}^d \partial_j \abs{u_i (x)}^2 \, \partial_j s(x)
= \frac{1}{2 s(x)} \sum_{j = 1}^d \partial_j \abs{\uu(x)}^2 \, \partial_j s(x) \\
& = \frac{1}{2 s(x)} \sum_{j = 1}^d \partial_j \big(s(x)^2\big) \, \partial_j s(x)
= \sum_{j = 1}^d \big(\partial_j s(x)\big)^2
= \abs{\grad s(x)}^2.
\end{split}
\end{equation*}
This concludes the proof.
\end{proof}

\subsection{Lim-sup inequality: Consistency} \label{sec:limsup}

We start with two results from~\cite{nwz2017} that we state without proofs. The
first one shows that the degree of orientation $s$ can be truncated near the end
points of the domain of definition $(-1/(d-1),1)$ of $\psi$ without increasing
the energy $E[s,\nn]$. We refer to \cite[Lemma~3.1]{nwz2017} for a proof.

\begin{lemma}[truncation of $s$] \label{lem:truncation}
Let the assumptions~\eqref{eq:DataAwayFromBoundaryS} and~\eqref{eq:DoubleWellAwayFromBoundaryS} hold.
Let $(s,\nn,\uu) \in \A(g,\rr)$.
For all $0<\rho \le \delta_0$, define
\begin{equation*}
s_\rho(x) := \min \bigg\{ 1 - \rho , \max \Big\{ - \frac{1}{d-1} + \rho , s(x) \Big\} \bigg \}
\quad
\text{and}
\quad
\uu_\rho(x) := s_\rho(x) \nn(x)
\quad
\text{for a.e. } x \in \Omega.
\end{equation*}
Then, $(s_{\rho}, \nn, \uu_\rho) \in \A(g,\rr)$ and
$E_1[s_{\rho}, \nn]\le E_1[s_{\rho}, \nn]$, $E_2[s_{\rho}] \le E_2[s]$.
\end{lemma}

\noindent
A simple consequence of Lemma \ref{lem:truncation}, based on
the characteristic function $\chi_{\{s_\rho\ne s\}} \to_{\rho\to0} \chi_\Omega$, is that
$\norm[H^1(\Omega)^{1+d}]{(s,\uu) - (s_{\rho},\uu_{\rho})}\to_{\rho\to0} 0$.
 The second result is about regularization of admissible functions but
preserving the structural condition \eqref{eq:structural} and boundary values.
This is a rather tricky two-scale process fully discussed in
\cite[Proposition~3.2]{nwz2017}.

\begin{lemma}[regularization of functions in $\A(g,\rr)$] \label{lem:regularization}
Let the assumptions~\eqref{eq:DataAwayFromBoundaryS} and~\eqref{eq:DataAwayFrom0} hold,
and suppose that $\Gamma_D = \partial\Omega$.
Let $(s,\nn,\uu) \in \A(g,\rr)$ and $\rho \le \delta_0$ such that
\begin{equation*}
- \frac{1}{d-1} + \rho
\le s(x)
\le 1 - \rho
\quad
\text{for a.e. } x \in \Omega.
\end{equation*}
Then, for all $\delta>0$,
there exists a triple $(s_{\delta},\nn_{\delta},\uu_{\delta}) \in \A(g,\rr)$
such that
$s_{\delta} \in W^{1,\infty}(\Omega)$
and
$\uu_{\delta} \in \WW^{1,\infty}(\Omega)$.
Moreover, there holds 
$\norm[H^1(\Omega)^{1+d}]{(s,\uu) - (s_{\delta},\uu_{\delta})} \le \delta$,
$\norm[\LL^2(\Omega\setminus\Sigma)]{\nn - \nn_\delta} \le \delta$,
and
\begin{equation*}
- \frac{1}{d-1} + \rho
\le s_{\delta}(x)
\le 1 - \rho
\quad
\text{for all } x \in \Omega.
\end{equation*}
\end{lemma}

It is well known that the Lagrange interpolation operator $I_h:C(\overline{\Omega})\to V_h$ is not stable in $H^1(\Omega)$ unless $d=1$. We exploit stability in
$L^\infty(\Omega)$ to derive stability in $W^{1,p}(\Omega)$ for $p>d$.
\begin{lemma}[$W^{1,p}$-stability of Lagrange interpolant]\label{lem:stab-Ih}
Let $v \in W^{1,p}(\Omega)$ for $d< p \le \infty$. Then
\begin{equation}\label{eq:stab-Lagrange}
  \|\grad I_h v\|_{L^p(K)} \le C \|\grad v\|_{L^p(K)}
  \quad\text{for all }\, K\in\mesh,
\end{equation}
where $C>0$ depends only
on the shape-regularity
of $\{ \mesh \}$.
\end{lemma}
\begin{proof}
Let $K\in\mesh$ be an arbitrary element and let $\overline{v}_K=\fint_K v$. An inverse estimate gives
\begin{equation*}
  \|\grad I_h v\|_{L^p(K)}^p \le |K| \, \|\grad I_h (v - \overline{v}_K)\|_{L^\infty(K)}^p
  \lesssim h_K^{d-p} \|v - \overline{v}_K\|_{L^\infty(K)}^p.
\end{equation*}
The Bramble--Hilbert estimate yields
$\|v - \overline{v}_K\|_{L^\infty(K)} \lesssim h_K^{1-d/p} \|\grad v\|_{L^p(K)}$ and
ends the proof.
\end{proof}

Applying a standard density argument in $W^{1,p}(\Omega)$, for $d<p<\infty$,
we deduce
\begin{equation}\label{eq:conv-Lagrange}
\lim_{h\to0} \|\grad (v - I_h v)\|_{L^p(\Omega)} = 0 \quad\textrm{for all } \, v \in W^{1,p}(\Omega).
\end{equation}

We have collected all the ingredients to show the existence of a recovery sequence.

\begin{proof}[\bf Proof of Theorem~\ref{thm:gamma_convergence}{\rm(i)}]
For the sake of clarity, we decompose the proof into seven steps.

\medskip
\emph{Step~1: Setup.}
Let $(s,\nn,\uu)\in\A(g,\rr)$.
For all $k \in \N$ such that $1/k \le \delta_0$,
let $0 < \delta_k \le 1/k$ be sufficiently small.
Applying successively Lemma~\ref{lem:truncation} (with $\rho = 1/k$)
and Lemma~\ref{lem:regularization} (with $\delta = \delta_k$),
we obtain $(s_k,\nn_k,\uu_k)\in\A(g,\rr)$
satisfying $(s_k, \uu_k) \in [W^{1,\infty}(\Omega)]^{1+d}$
and 
$- 1/(d-1) + 1/k \le s_k \le 1 - 1/k$ in $\Omega$ for all $k$.
Moreover, we have that
\[
\norm[H^1(\Omega)^{1+d}]{(s,\uu) - (s_k,\uu_k)} \to 0,
\qquad
\norm[\LL^2(\Omega\setminus\Sigma)]{\nn - \nn_k} \to 0.
\]

Since $(s,\nn,\uu) \in \A(g,\rr)$,
Proposition~\ref{prop:L2diff} guarantees that
$\nn$ is $L^2$-differentiable a.e.\ in $\Omega \setminus \Sigma$,
with its $L^2$-gradient given by~\eqref{eq:L2gradient_formula}
and that the identity~\eqref{eq:structural2} holds.
The same result is valid for $\nn_k$ a.e.\ in $\Omega \setminus \Sigma_k$,
where $\Sigma_k := \{ x \in \Omega : s_k(x) = 0\}$.

Let $s_{k,h} := I_h[s_k]$ and $\uu_{k,h} :=I_h[\uu_k]$.
Let $\nn_{k,h}\in\VVV_h$ be defined, for all $z\in\nodes$, as
\begin{equation*}
\nn_{k,h}(z)
:=\begin{cases} 
\uu_{k,h}(z)/s_{k,h}(z)=\uu_k(z)/s_k(z) & \text{if } z \in \Omega\setminus\Sigma_k, \\
\text{an arbitrary unit vector} & \text{if } z \in \Sigma_k.
\end{cases}
\end{equation*}
Note that,
by construction,
$(s_{k,h},\nn_{k,h},\uu_{k,h})$ satisfies the discrete structural condition~\eqref{eq:structural_h}, and $\|\nn_{k,h}\|_{L^\infty(\Omega)} \le C$.
Moreover, since
$0 = \norm[L^1(\Omega)]{I_h \big[\abs{\nn_{k,h}}^2 \big]-1} \le \eps$
as well as $s_{k,h}(z) = g_h(z)$ and $\uu_{k,h}(z) = \rr_h(z)$ for all $z \in \nodes \cap \Gamma_D$, we deduce $(s_{k,h},\nn_{k,h},\uu_{k,h})\in\A_{h,\eps}(g_h,\rr_h)$.

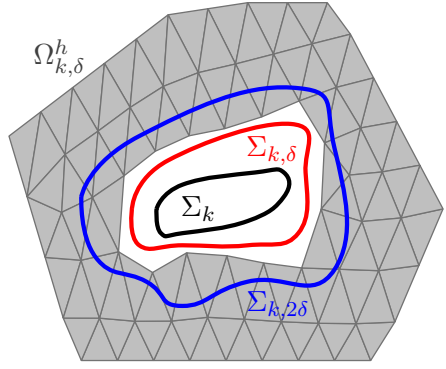
\begin{figure}[h]
\begin{tikzpicture}
\begin{axis}[
axis lines = none
]
\addplot[patch,
color=gray!50!white, 
faceted color = gray!90!black,
line width = 0.5pt,
patch table = {./data/elements.dat}] file{./data/coordinates.dat};
\draw[ultra thick, blue] plot[smooth, tension=.9] coordinates {(0,0.75) (0.40,1.0) (1.45,1.2) (1.75,0.95) (1.75,0.45) (1.15,0.4) (0.70,0.3) (0.45,0.40) (0.1,0.5) (0,0.75)};
\draw[ultra thick, red] plot[smooth, tension=.8] coordinates {(0.35,0.70) (0.60,0.90) (1.45,1.05) (1.55,0.85) (1.50,0.60) (1.15,0.55) (0.70,0.55) (0.40,0.55) (0.35,0.70)};
\draw[ultra thick] plot[smooth, tension=.8] coordinates {(0.52,0.70) (0.70,0.8) (1.20,0.85) (1.4,0.85) (1.40,0.75) (1.15,0.65) (0.70,0.6) (0.55,0.6) (0.52,0.70)};
\node[anchor=south east] (x) at (1.0,0.61) {$\Sigma_k$};
\node[anchor=south east] (x) at (1.57,0.84) {\textcolor{red}{$\Sigma_{k,\delta}$}};
\node[anchor=south east] (x) at (1.65,0.2) {\textcolor{blue}{$\Sigma_{k,2\delta}$}};
\node[anchor=south east] (x) at (0.1,1.2) {\textcolor{gray!50!black}{$\Omega_{k,\delta}^h$}};
\end{axis}
\end{tikzpicture}
\caption{A schematic illustration of the mutual relations of the sets defined in Step~1
of the proof of Theorem~\ref{thm:gamma_convergence}{\rm(i)}.
Note that the set $\Sigma_k \subset \Omega$ is closed,
as it is the preimage of a closed set with respect to the continuous function $s_k$,
but it might be more topologically complicated than in the picture.}
\label{fig:sets_limsup}
\end{figure}

Given $\delta > 0$,
we consider the sets
\begin{equation*}
\Sigma_{k,\delta} := \{ x \in \Omega : \abs{s_k(x)} \le \delta \}
\quad
\text{and}
\quad
\Omega_{k,\delta}^h
:= \bigcup \{K \in \mesh : K \cap \Sigma_{k,\delta} = \emptyset\}.
\end{equation*}
Note that, by construction, there holds
$\Omega_{k,\delta}^h \subset \Omega \setminus \Sigma_{k,\delta}$; we refer to Figure \ref{fig:sets_limsup}.

Let $K \in \mesh$ such that $K \cap \Sigma_{k,\delta} \neq \emptyset$.
In particular, there exists $x_0 \in K \cap \Sigma_{k,\delta}$.
For $x_1 \in K$ arbitrary,
Lipschitz continuity of $s_k$ yields
\begin{equation*}
\abs{s_k(x_1)}
\le \abs{s_k(x_0)} + \abs{s_k(x_1) - s_k(x_0)}
\le \delta + C_kh.
\end{equation*}
In particular, $\Omega \setminus \Omega_{k,\delta}^h \subset \Sigma_{k,2\delta}$
provided $h$ is sufficiently small so that $C_k h\le\delta$; see Figure~\ref{fig:sets_limsup}.

Now, for any $x \in \Omega_{k,\delta}^h$, we infer that
\begin{equation*}
\abs{s_k(x) - s_{k,h}(x)}
\le \norm[L^{\infty}(\Omega_{k,\delta}^h)]{s_k - s_{k,h}}
= \norm[L^{\infty}(\Omega_{k,\delta}^h)]{s_k - I_h[s_k]}
\lesssim h \norm[L^{\infty}(\Omega_{k,\delta}^h)]{\grad s_k},
\end{equation*}
whence
\begin{equation*}
\abs{s_{k,h}(x)}
\ge \abs{s_k(x)} - \abs{s_k(x) - s_{k,h}(x)}
> \delta - C h \norm[L^{\infty}(\Omega_{k,\delta}^h)]{\grad s_k}
> \delta/2
\end{equation*}
provided the mesh size $h$ is chosen to be sufficiently small.
Hence, for those $h$, we can define
$\wt\nn_{k} := \uu_{k,h} / s_{k,h}$ in $\Omega_{k,\delta}^h$.
Note that, by definition, the relation $\nn_{k,h} = I_h[\wt\nn_k]$ in $\Omega_{k,\delta}^h$
holds.

To conclude this step, we observe that the $L^2$-gradient $\Grad \nn_{k}$
  of $\nn_{k}$ exists a.e.\ in $\Omega\setminus\Sigma_k$ and
\begin{equation} \label{eq:auxiliary_gradN}
\int_{\Omega_{k,\delta}^h}\abs{\Grad \nn_{k} - \Grad \nn_{k,h}}^2
\lesssim
\int_{\Omega_{k,\delta}^h}\abs{\Grad \nn_{k} - \Grad \wt\nn_{k}}^2
+
\int_{\Omega_{k,\delta}^h}\abs{\Grad \wt\nn_{k} - \Grad \nn_{k,h}}^2,
\end{equation}
where $\Grad \wt\nn_{k}$ and $\Grad \nn_{k,h}$ denote
the weak gradients of $\wt\nn_{k}$ and $\nn_{k,h}$, respectively, which
coincide elementwise with their classical gradients in $\Omega_{k,\delta}^h$.
In the following steps,
we will show that,
for fixed $k \in \N$ and $\delta>0$,
both two terms on the right-hand side of~\eqref{eq:auxiliary_gradN} converge to $0$ as $h \to 0$.

\medskip
\emph{Step~2: Proof of}
$\lim_{h\to0} \int_{\Omega_{k,\delta}^h}\abs{\wt\nn_{k} - \nn_{k,h}}^2 + \abs{\Grad \wt\nn_{k} - \Grad \nn_{k,h}}^2 = 0.$ Since $\nn_{k,h} = I_h[\wt\nn_{k}]$
in $\Omega_{k,\delta}^h$,
a classical local interpolation estimate yields that
\begin{equation*}
\int_{\Omega_{k,\delta}^h}\abs{\Grad \wt\nn_{k} - \Grad \nn_{k,h}}^2
=
\sum_{\substack{K \in \mesh \\ K \cap \Sigma_{k,\delta} = \emptyset}}
\int_K \big\lvert \Grad (\wt\nn_{k} - I_h [\tilde\nn_{k}]) \big\rvert^2
\lesssim
\sum_{\substack{K \in \mesh \\ K \cap \Sigma_{k,\delta} = \emptyset}}
h_K^2 \big\lVert D^2 \wt\nn_{k} \big\rVert_{\LL^2(K)}^2.
\end{equation*}
Moreover, in view of $\wt\nn_k = \uu_{k,h} / s_{k,h}$ in $\Omega_{k,\delta}^h$,
explicit computations reveal that
\begin{align*}
\partial_i \wt\nn_{k}
& =
s_{k,h}^{-1} \, \partial_i \uu_{k,h} - s_{k,h}^{-2} \, \partial_i s_{k,h} \, \uu_{k,h} 
=
s_{k,h}^{-1} \big( \partial_i \uu_{k,h} - \partial_i s_{k,h} \, \wt\nn_{k} \big),\\
\partial_j \partial_i \wt\nn_{k}
& =
s_{k,h}^{-1} \big(
s_{k,h}^{-1} \, \partial_j s_{k,h} \, \partial_i s_{k,h} \, \wt\nn_k
- \partial_i s_h \, \partial_j \wt\nn_k
- s_{k,h}^{-1} \, \partial_j s_{k,h} \, \partial_i \uu_{k,h}
\big),
\end{align*}
for all $1\le i,j \le d$.
Several applications of the generalized H\"older inequality,
in conjunction with the lower bound $|s_{k,h}|>\delta/2$ in $\Omega_{k,\delta}^h$, thus yield
\begin{equation*}
\begin{split}
\big\lVert D^2 \wt\nn_{k} \big\rVert_{\LL^2(K)}
& \lesssim
\delta^{-3} \norm[L^8(K)]{\grad s_{k,h}}^2 \norm[\LL^4(K)]{\uu_{k,h}} \\
& \quad + \delta^{-1} \norm[\LL^4(K)]{\grad s_{k,h}} \big( \delta^{-1} \norm[\LL^4(K)]{\Grad \uu_{k,h}}
+ \delta^{-2} \norm[\LL^8(K)]{\uu_{k,h}} \norm[\LL^8(K)]{\grad s_{k,h}} \big) \\
& \quad + \delta^{-2} \norm[\LL^4(K)]{\grad s_{k,h}} \norm[\LL^4(K)]{\Grad \uu_{k,h}}.
\end{split}
\end{equation*}
In view of \eqref{eq:stab-Lagrange},
$s_{k,h}$ (resp., $\uu_{k,h}$) is uniformly bounded in $W^{1,p}(\Omega)$ (resp., $\WW^{1,p}(\Omega)$) when $d<p\le\infty$.
Altogether,
we thus obtain the desired estimate
\begin{equation*}
\int_{\Omega_{k,\delta}^h}\abs{\Grad \wt\nn_{k} - \Grad \nn_{k,h}}^2
  + \sum_{\substack{K \in \mesh \\ K \cap \Sigma_{k,\delta} = \emptyset}}h_K^{-2} \int_{K} \abs{\wt\nn_{k} - \nn_{k,h}}^2
\lesssim
\sum_{\substack{K \in \mesh \\ K \cap \Sigma_{k,\delta} = \emptyset}}
h_K^2 \big\lVert D^2 \wt\nn_{k} \big\rVert_{\LL^2(K)}^2
\lesssim
h^2.
\end{equation*}

\medskip
\emph{Step~3: Proof of} 
$\lim_{h\to0}  \int_{\Omega_{k,\delta}^h}\abs{\nn_{k} - \wt\nn_{k}}^2 + \abs{\Grad \nn_{k} - \Grad \wt\nn_{k}}^2 = 0$.
We first observe that
\begin{equation*}
\begin{split}
\norm[\LL^p(\Omega_{k,\delta}^h)]{\wt\nn_{k} - \nn_{k}}
& =
\norm[\LL^p(\Omega_{k,\delta}^h)]{s_{k,h}^{-1}\uu_{k,h} - s_{k}^{-1}\uu_{k}} \\
& \lesssim
\delta^{-2}
\norm[L^p(\Omega_{k,\delta}^h)]{s_{k} - s_{k,h}}
\norm[\LL^\infty(\Omega_{k,\delta}^h)]{\uu_{k,h}}
+ \delta^{-1} \norm[\LL^p(\Omega_{k,\delta}^h)]{\uu_{k,h} - \uu_{k}},
\end{split}
\end{equation*}
for all $p \ge 1$. This shows, in view of \eqref{eq:conv-Lagrange}, that $\norm[\LL^p(\Omega_{k,\delta}^h)]{\wt\nn_{k} - \nn_{k}} \to 0$ as $h \to 0$ for $d<p<\infty$. To deal with the gradient part, we resort to 
available expressions of $\Grad \nn_{k}$ and $\Grad \wt\nn_{k}$ to write
\begin{equation*}
\int_{\Omega_{k,\delta}^h}\!\abs{\Grad \nn_{k} - \Grad \wt\nn_{k}}^2
=
\int_{\Omega_{k,\delta}^h}\!\abs{s_{k}^{-1}(\Grad \uu_{k} - \nn_{k}\otimes\grad s_{k})  - s_{k,h}^{-1}(\Grad \uu_{k,h} - \wt\nn_{k}\otimes\grad s_{k,h})}^2 \le T_1 + T_2 + T_3,
\end{equation*}
where
\begin{align*}
T_1 & :=
\int_{\Omega_{k,\delta}^h}\abs{s_{k,h}^{-1}(\Grad \uu_{k} - \Grad \uu_{k,h})}^2,
\\
T_2 & :=
\int_{\Omega_{k,\delta}^h}\abs{s_{k,h}^{-1}(\wt\nn_{k}\otimes\grad s_{k,h} - \nn_{k}\otimes\grad s_{k})}^2,
\\
T_3 & := \int_{\Omega_{k,\delta}^h}\abs{(s_{k}^{-1} -  s_{k,h}^{-1})(\Grad \uu_{k} - \nn_{k}\otimes\grad s_{k})}^2.
\end{align*}
Recalling again $|s_k|,|s_{k,h}| > \delta / 2$ in $\Omega_{k,\delta}^h$, as well as
\eqref{eq:conv-Lagrange}, the asserted estimate follows from
\begin{align*}
T_1
& \lesssim \delta^{-2} \norm[\LL^2(\Omega)]{\Grad (\uu_{k} - I_h\uu_{k})}^2,
\\
T_2
& \lesssim \delta^{-2} \norm[\LL^4(\Omega_{k,\delta}^h)]{\grad s_{k,h}}^2 \norm[\LL^4(\Omega_{k,\delta}^h)]{\wt\nn_{k} - \nn_{k}}^2
+ \delta^{-2} \norm[\LL^4(\Omega_{k,\delta}^h)]{\nn_{k}}^2 \norm[\LL^4(\Omega_{k,\delta}^h)]{\grad (s_k - I_h s_k)}^2,
\\
T_3
& \lesssim
\delta^{-4} \norm[L^{4}(\Omega)]{s_{k} - I_hs_{k}}^2
\norm[\LL^4(\Omega)]{\Grad \uu_{k} - \nn_{k}\otimes\grad s_{k}}^2.
\end{align*}

\medskip
\emph{Step~4: Proof of}
$\lim_{h\to0} \int_{\Omega} s_{k,h}^2\abs{\Grad\nn_{k,h}}^2 = \int_{\Omega\setminus\Sigma_k}s_k^2\abs{\Grad\nn_k}^2$. Combining Steps 2--3 gives
\begin{equation}\label{eq:conv_nkh}
\lim_{h \to 0} \int_{\Omega_{k,\delta}^h}\abs{\Grad \nn_{k} - \Grad \nn_{k,h}}^2 = 0.
\end{equation}
In order to exploit this property, we split the integral under consideration as
\begin{equation} \label{eq:auxiliary_energy1}
\int_{\Omega} s_{k,h}^2\abs{\Grad\nn_{k,h}}^2
=
\int_{\Omega_{k,\delta}^h} s_{k,h}^2\abs{\Grad\nn_{k,h}}^2
+ \int_{\Omega\setminus \Omega_{k,\delta}^h} s_{k,h}^2\abs{\Grad\nn_{k,h}}^2.
\end{equation}
The fact that $s_{k,h} \to s_{k}$ strongly in $L^p(\Omega)$ as $h \to 0$ for
$d<p<\infty$, according to \eqref{eq:conv-Lagrange}, together with
$s_{k,h} \in L^\infty(\Omega)$ uniformly in $h$, $\Grad\nn_{k}\in L^{\infty}(\Omega\setminus\Sigma_{k,\delta})$ and \eqref{eq:conv_nkh}, yields
\begin{equation*}
\lim_{h \to 0} 
\bigg\lvert
\int_{\Omega_{k,\delta}^h} s_{k,h}^2\abs{\Grad\nn_{k,h}}^2
- \int_{\Omega_{k,\delta}^h} s_{k}^2\abs{\Grad\nn_{k}}^2
\bigg\rvert
= 0.
\end{equation*}
Since
$\Omega \setminus \Sigma_{k,2\delta} \subset \Omega_{k,\delta}^h \subset \Omega \setminus \Sigma_{k,\delta}$,
we deduce
\begin{equation*}
\begin{split}
\lim_{\delta\to0} \lim_{h \to 0} \int_{\Omega_{k,\delta}^h} s_{k,h}^2\abs{\Grad\nn_{k,h}}^2
& =
\int_{\Omega \setminus \Sigma_k} s_{k}^2\abs{\Grad\nn_{k}}^2.
\end{split}
\end{equation*}
Now, we consider the second term on the right-hand side of~\eqref{eq:auxiliary_energy1}.
Since $\Omega\setminus \Omega_{k,\delta} \subset \Sigma_{k,2\delta}$ and
$s_{k,h} \Grad\nn_{k,h} = \Grad(s_{k,h} \nn_{k,h}) - \nn_{k,h} \otimes \grad s_{k,h}$,
using $\uu_{k,h}=I_h(s_{k,h} \nn_{k,h})$, we see that
\begin{equation*}
\begin{split}
\int_{\Omega\setminus \Omega_{k,\delta}^h} s_{k,h}^2\abs{\Grad\nn_{k,h}}^2
& 
\lesssim
\int_{\Sigma_{k,2\delta}} \abs{\Grad(s_{k,h} \nn_{k,h})}^2
+
\int_{\Sigma_{k,2\delta}} \abs{\nn_{k,h} \otimes \grad s_{k,h}}^2 \\
& \le
\int_{\Sigma_{k,2\delta}} \abs{\Grad(s_{k,h} \nn_{k,h}) - \Grad I_h(s_{k,h} \nn_{k,h})}^2
+
\int_{\Sigma_{k,2\delta}} \abs{\Grad\uu_{k,h}}^2
+
\int_{\Sigma_{k,2\delta}} \abs{\grad s_{k,h}}^2.
\end{split}
\end{equation*}
Combining an interpolation estimate with the fact that $s_{k,h}$ and $\nn_{k,h}$ are piecewise affine, and exploiting an inverse estimate to bound 
$\norm[\LL^{\infty}(K)]{\Grad \nn_{k,h}}$ in terms of $\norm[\LL^{\infty}(K)]{\nn_{k,h}}\le C$, yields
\begin{align*}
  \int_{\Sigma_{k,2\delta}} \abs{\Grad(s_{k,h} \nn_{k,h}) &- \Grad I_h(s_{k,h} \nn_{k,h})}^2
  \lesssim 
  \sum_{\substack{K \in \mesh \\ K \cap \Sigma_{k,2\delta} \neq \emptyset}} h_K^2 \norm[\LL^2(K)]{D^2(s_{k,h} \nn_{k,h})}^2
  \\
  &\lesssim
\sum_{\substack{K \in \mesh \\ K \cap \Sigma_{k,2\delta} \neq \emptyset}}
h_K^2 \norm[\LL^2(K)]{\grad s_{k,h}}^2 \norm[\LL^{\infty}(K)]{\Grad \nn_{k,h}}^2
\lesssim \sum_{\substack{K \in \mesh \\ K \cap \Sigma_{k,2\delta} \neq \emptyset}}
\norm[\LL^2(K)]{\grad s_{k,h}}^2.
\end{align*}
Using the $W^{1,p}$-stability \eqref{eq:stab-Lagrange} of the nodal interpolant
with $p>d$ for elements $K \cap \Sigma_{k,2\delta} \neq \emptyset$, we end up with the following as $\delta\to0$
\begin{equation*}
\int_{\Omega\setminus \Omega_{k,\delta}^h} s_{k,h}^2\abs{\Grad\nn_{k,h}}^2
\lesssim
\norm[\LL^p(\Sigma_{k,3\delta})]{\Grad\uu_{k}}^2
+ \norm[\LL^p(\Sigma_{k,3\delta})]{\grad s_{k}}^2
\to
\norm[\LL^p(\Sigma_{k})]{\Grad\uu_{k}}^2
+ \norm[\LL^p(\Sigma_{k})]{\grad s_{k}}^2
= 0,
\end{equation*}
because the Lipschitz continuity of $s_k$ implies $K\subset\Sigma_{k,3\delta}$
provided $h$ is sufficiently small.

\medskip
\emph{Step~5: Proof of} 
$\lim_{h\to0} \int_{\Omega} \abs{\nn_{k,h} \otimes \grad s_{k,h}}^2
= \int_{\Omega} \abs{\grad s_{k}}^2$.
We split the integral as
\begin{equation*}
\int_{\Omega} \abs{\nn_{k,h} \otimes \grad s_{k,h}}^2
=
\int_{\Omega_{k,\delta}^h} \abs{\nn_{k,h} \otimes \grad s_{k,h}}^2
+
\int_{\Omega\setminus\Omega_{k,\delta}^h} \abs{\nn_{k,h} \otimes \grad s_{k,h}}^2.
\end{equation*}
Exploiting the identity
$\nn_{k,h} \otimes \grad s_{k,h} -\nn_{k}\otimes \grad s_{k}
=(\nn_{k,h}  - \nn_{k})\otimes\grad s_{k} + \nn_{k,h} \otimes (\grad s_{k,h} - \grad s_{k})$,
and using the convergence results for $s_{k,h}$ and $\nn_{k,h}$ in $\Omega_{k,\delta}^h$
from Steps 1--3, we readily see that
\begin{equation*}
\lim_{h\to0}\int_{\Omega_{k,\delta}^h} \abs{\nn_{k,h} \otimes \grad s_{k,h}}^2
=
\int_{\Omega\setminus\Sigma_{k,\delta}} \abs{\nn_{k} \otimes \grad s_{k}}^2
\to
\int_{\Omega\setminus\Sigma_{k}} \abs{\nn_{k} \otimes \grad s_{k}}^2
=\int_{\Omega} \abs{\grad s_{k}}^2,
\end{equation*}
as $\delta\to0$.
Moreover, employing $\Omega\setminus \Omega_{k,\delta} \subset \Sigma_{k,2\delta}$
together with \eqref{eq:stab-Lagrange} implies
\begin{equation*}
\int_{\Omega\setminus\Omega_{k,\delta}^h} \abs{\nn_{k,h} \otimes \grad s_{k,h}}^2
\lesssim
\norm[\LL^2(\Sigma_{k,2\delta})]{\grad I_h s_{k}}^2
\lesssim
\norm[\LL^p(\Sigma_{k,2\delta})]{\grad I_h s_{k}}^2
\lesssim
\norm[\LL^p(\Sigma_{k,3\delta})]{\grad s_{k}}^2.
\end{equation*}
Finally, taking $\delta \to 0$ yields
$\norm[\LL^p(\Sigma_{k,3\delta})]{\grad s_{k}}\to
\norm[\LL^p(\Sigma_k)]{\grad s_{k}} = 0$, which is the desired limit.

\medskip

\emph{Step~6: Convergence of} $\{ s_{k,h}\}$, $\{ \nn_{k,h}\}$, \emph{and} $\{ \uu_{k,h}\}$.
The triangle inequality gives 
\begin{equation*}
\norm[H^1(\Omega)]{s_{k,h} - s}
\le
\norm[H^1(\Omega)]{s_{k,h} - s_k}
+
\norm[H^1(\Omega)]{s_k - s}
\to 0
\quad
\text{as } h \to 0
\text{ and } k \to \infty.
\end{equation*}
Likewise, $\uu_{k,h}\to\uu$ in $\HH^1(\Omega)$ as $h\to0$ and
$k\to\infty$. Turning to $\nn$, we observe that
\begin{equation*}
\norm[\LL^2(\Omega\setminus\Sigma)]{\nn_{k,h} - \nn_k}
\lesssim
\norm[\LL^2(\Omega_{k,\delta}^h)]{\nn_{k,h} - \nn_k}
+
\norm[\LL^2(\Sigma_{k,2\delta}\setminus\Sigma)]{\nn_{k,h} - \nn_k},
\end{equation*}
and $\norm[\LL^2(\Omega_{k,\delta}^h)]{\nn_{k,h} - \nn_k}\to0$ as
$h \to 0$ from Steps~2--3. Instead, for the second term we have
\begin{equation*}
\norm[\LL^2(\Sigma_{k,2\delta}\setminus\Sigma)]{\nn_{k,h} - \nn_k}
\le 2 \abs{\Sigma_{k,2\delta} \setminus \Sigma}^{1/2}
\to 0
\quad
\text{as }\delta \to 0
\text{ and } k \to \infty.
\end{equation*}
The convergence of $\nn_{k,h}$ to $\nn$ in $\LL^2(\Omega\setminus\Sigma)$ then follows from the triangle inequality.

\medskip

\emph{Step~7: Convergence of energy}.
The previous steps yield
\begin{equation*}
  \lim_{h \to 0} E_1^h[s_{k,h},\nn_{k,h}] = E_1[s_k,\nn_k] =
  \frac12 \int_{\Omega\setminus\Sigma_k} \kappa|\grad s_k|^2 + s_k^2 |\Grad\nn_k|^2.
\end{equation*}
To prove that $E_1[s_k,\nn_k] \to E_1[s,\nn]$ as $k \to \infty$ we
resort to \eqref{eq:equiv-energy}, namely
\[
E_1[s_k,\nn_k]=\wt{E}_1[s_k,\uu_k] = \frac12 \int_\Omega (\kappa-1)|\grad s_k|^2
  + |\Grad\uu_k|^2 \to \frac12 \int_\Omega (\kappa-1)|\grad s|^2
  + |\Grad\uu|^2 = E_1[s,\nn].
\]
We now deal with $E_2$.
Since $- 1/(d-1) + 1/k \le s_k \le 1 - 1/k$ in $\Omega$,
assumption~\eqref{eq:DoubleWellAwayFromBoundaryS} guarantees that
$0 \le \psi(s_{k,h}) \le \max\{ \psi(-1/(d-1) + 1/k) , \psi(1- 1/k) \}$.
Hence, the dominated convergence theorem implies that
\begin{equation*}
\lim_{h \to 0} E_2^h[s_{k,h}]
=
\lim_{h \to 0} \int_{\Omega} \psi(I_h s_k)
=
\int_{\Omega} \lim_{h \to 0} \psi(I_h s_k)
= \int_{\Omega} \psi(s_{k})
= E_2[s_{k}].
\end{equation*}
Moreover, the monotonicity of $\psi$ in $(-1/(d-1),-1/(d-1)+\delta_0)$ and in $(1-\delta_0,1)$ translates into $\psi(s_k)\ge 0$ increasing and converging pointwise to $\psi(s)$, whence the monotone convergence theorem gives
\[  
E_2[s_k] = \int_\Omega \psi(s_k) \to \int_\Omega \psi(s) = E_2 [s].
\]
Consequently, the sequence
$(s_h,\nn_h,\uu_h) := (s_{k,h},\nn_{k,h},\uu_{k,h})\in\A_{h,\eps}(g_h,\rr_h)$ for $k$
sufficiently large depending on $h$
converges to $(s,\nn,\uu)$ in
$H^1(\Omega) \times \LL^2(\Omega\setminus\Sigma) \times \HH^1(\Omega)$
as $h \to 0$ and satisfies

\begin{equation*}
\lim_{h \to 0} E^h[s_h,\nn_h] = E[s,\nn].
\end{equation*}
This implies the lim-sup inequality~\eqref{eq:limsup_inequality} and concludes the proof.
\end{proof}

\subsection{Lim-inf inequality: Stability} \label{sec:liminf}

To show the lim-inf inequality,
we first prove that admissible discrete pairs $(s_h,\nn_h)$
with uniformly bounded energy
are uniformly bounded in $H^1$. In constrast to~\cite{nwz2017},
we do not need to assume that $\mesh$ is weakly acute.

\begin{lemma}[coercivity] \label{lem:coercivity}
Let $\{(s_h,\nn_h,\uu_h)\} \subset V_h \times \VVV_h \times \VVV_h$ satisfy
$\uu_h = I_h[s_h \nn_h]$ and $\abs{\nn_h(z)} \geq 1$ for all $z \in \nodes$.
Then, there exists a constant $C>0$ depending only on the shape-regularity
of $\{ \mesh \}$ and $\kappa$ such that
\begin{equation*}
  C \max\Big\{ \norm[\LL^2(\Omega)]{\Grad\uu_h}^2,
  \norm[\LL^2(\Omega)]{\Grad(s_h \nn_h)}^2,
  \norm[\LL^2(\Omega)]{\grad s_h}^2 \Big\}
  \le E_1^h[s_h,\nn_h].
\end{equation*}
\end{lemma}

\begin{proof}
Since $\norm[\LL^{\infty}(K)]{\nn_h} \geq 1$ for all $K \in \mesh$
and $\grad s_h$ is piecewise constant,
it holds that
\begin{equation*}
\begin{split}
\norm[\LL^2(\Omega)]{\grad s_h}^2
& \leq \sum_{K \in \mesh} \norm[\LL^{\infty}(K)]{\nn_h}^2 \norm[\LL^2(K)]{\grad s_h}^2
= \sum_{K \in \mesh} \abs{K} \, \norm[\LL^{\infty}(K)]{\nn_h}^2 \abs{\grad s_h\vert_K}^2 \\
& 
\lesssim \sum_{K \in \mesh} \abs{\grad s_h\vert_K}^2 \norm[\LL^2(K)]{\nn_h}^2
= \norm[\LL^2(\Omega)]{\nn_h \otimes \grad s_h}^2
\le \frac{2}{\kappa} E_1^h[s_h,\nn_h],
\end{split}
\end{equation*}
where the hidden multiplicative constant depends only
on the shape-regularity of $\{ \mesh \}$.
Let $\wt{\uu}_h = s_h \nn_h$ and use \eqref{eq:stab-Lagrange} for $p>d$
  in conjunction with an inverse estimate to obtain
\begin{align*}
  \norm[\LL^2(K)]{\grad I_h\wt{\uu}_h}\lesssim
  \abs{K}^{\frac{p-2}{2p}} \norm[\LL^p(K)]{\grad I_h\wt{\uu}_h}
  \lesssim \abs{K}^{\frac{p-2}{2p}} \norm[\LL^p(K)]{\grad \wt{\uu}_h}
  \lesssim \norm[\LL^2(K)]{\grad \wt{\uu}_h}
  \quad\text{for all } K\in\mesh.
\end{align*}
Consequently, for $\uu_h = I_h\wt{\uu}_h$ we deduce
\begin{equation*}
  \norm[\LL^2(\Omega)]{\Grad\uu_h}^2 \lesssim
  \norm[\LL^2(\Omega)]{\Grad \wt{\uu}_h}^2 \lesssim
  \norm[\LL^2(\Omega)]{\nn_h \otimes \grad s_h}^2
  + \norm[\LL^2(\Omega)]{s_h \Grad \nn_h}^2
  \lesssim E_1^h[s_h,\nn_h].
\end{equation*}
This completes the proof.
\end{proof}

We are now ready to extract convergent subsequences
and characterize their limits.

\begin{lemma}[characterization of limits] \label{lem:limits}
Let $\{(s_h,\nn_h,\uu_h)\} \subset \A_{h,\eps}(g_h,\rr_h)$
be a sequence such that $E_1^h[s_h,\nn_h] \le C$ and $\norm[\LL^{\infty}(\Omega)]{\nn_h} \le C$,
where $C>0$ is a constant independent of $h$.
Then, there exist a triple
$(s,\nn,\uu) \in \A(g,\rr)$
and a subsequence (not relabeled)
of $\{(s_h,\nn_h,\uu_h)\}$
satisfying the following properties:
\begin{enumerate}[$\bullet$]
\item As $h \to 0$, $(s_h,\uu_h,s_h \nn_h)$ converges towards $(s,\uu,\uu)$ weakly in
$H^1(\Omega) \times \HH^1(\Omega) \times \HH^1(\Omega)$,
strongly in $L^2(\Omega) \times \LL^2(\Omega) \times \LL^2(\Omega)$,
and pointwise a.e.\ in $\Omega$;
\item $\nn_h$ converges towards $\nn$ strongly in $\LL^2(\Omega\setminus\Sigma)$
and pointwise a.e.\ in $\Omega\setminus\Sigma$ as $h \to 0$ and $\eps\to0$;
\item $\nn$ is $L^2$-differentiable a.e.\ in $\Omega\setminus\Sigma$
and the orthogonal decomposition $\abs{\Grad \uu}^2 = \abs{\grad s}^2 + s^2 \abs{\Grad \nn}^2$ is valid
a.e.\ in $\Omega\setminus\Sigma$,
\end{enumerate}
where $\Sigma \subset \Omega$ is given by~\eqref{eq:singular}.
\end{lemma}

\begin{proof}
For the sake of clarity, we divide the proof into 3 steps.

\medskip
\emph{Step~1: Convergence of $\{s_h\}$, $\{\uu_h\}$, and $\{ s_h \nn_h \}$}.
Since the energy $E_1^h[s_h,\nn_h]$ is uniformly bounded,
Lemma~\ref{lem:coercivity} (coercivity) gives uniform bounds in
$H^1(\Omega) \times \HH^1(\Omega) \times \HH^1(\Omega)$ for the 
the sequence $\{(s_h,\uu_h,s_h \nn_h)\}$.
With successive extractions of subsequences (not relabeled),
one can show that there exists a limit
$(s,\uu,\wt{\uu}) \in H^1(\Omega) \times \HH^1(\Omega) \times \HH^1(\Omega)$
such that $(s_h,\uu_h,s_h \nn_h)$ converges to $(s,\uu,\wt{\uu})$
weakly in $H^1(\Omega) \times \HH^1(\Omega)\times \HH^1(\Omega)$,
strongly in $L^2(\Omega) \times \LL^2(\Omega) \times \LL^2(\Omega)$,
and pointwise a.e.\ in $\Omega$.
Moreover, weak $H^1$-convergence guarantees attainment of traces, namely
$s=g$ and $\uu=\wt{\uu}=\rr$ on $\Gamma_D$.
To see this, note that $g_h=I_hg \to g$ in $W^{1,p}(\Omega)$ for $p>d$, according to \eqref{eq:conv-Lagrange}, and so in $H^1(\Omega)$. Therefore $s_h-g_h\in H^1_0(\Omega)$ satisfies
\[
s_h-g_h \weakto s-g \in H^1_0(\Omega),
\]
because $H^1_0(\Omega)$ is closed under weak convergence. Hence $s=g$ on $\Gamma_D$
in the sense of traces, as asserted. Dealing with $\uu_h$ and $\wt{\uu}_h$ is identical.
Since $\uu_h = I_h[s_h \nn_h]$, interpolation and inverse estimates, yield

\begin{equation*}
\begin{split}
\norm[\LL^2(\Omega)]{\uu_h - s_h \nn_h}^2
& \lesssim
\sum_{K \in \mesh} h_K^4 \norm[\LL^2(K)]{D^2(s_h\nn_h)}^2 \\
& \lesssim
\sum_{K \in \mesh} h_K^2 \norm[\LL^2(K)]{\Grad(s_h\nn_h)}^2\lesssim h^2 E_1^h[s_h,\nn_h] \leq C h^2.
\end{split}
\end{equation*}
This shows that $s_h \nn_h$ and $\uu_h$ converge strongly in $\LL^2(\Omega)$
towards the same limit i.e., $\wt \uu = \uu$.
Moreover, $s_h \nn_h$ converges to $\uu$
weakly in $\HH^1(\Omega)$ and pointwise a.e.\ in $\Omega$.

\medskip
\emph{Step~2: $\abs{s} = \abs{\uu}$ a.e.\ in $\Omega$}.
The triangle inequality yields
\begin{equation*}
\begin{split}
 \norm[L^1(\Omega)]{\abs{\uu_h}^2 - \abs{s_h}^2}  &\le
\norm[L^1(\Omega)]{\abs{\uu_h}^2 - I_h\big[\abs{\uu_h}^2\big]}
\\
& + \norm[L^1(\Omega)]{I_h\big[\abs{\uu_h}^2 - \abs{s_h}^2\big]}
+ \norm[L^1(\Omega)]{\abs{s_h}^2 - I_h\big[\abs{s_h}^2\big]}.
\end{split}
\end{equation*}
For the first and third terms on the right-hand side,
standard interpolation estimates yield
\begin{align*}
\norm[L^1(\Omega)]{\abs{s_h}^2 - I_h\big[\abs{s_h}^2\big]}
\lesssim h^2 \norm[\LL^2(\Omega)]{\grad s_h}^2,
\qquad
\norm[L^1(\Omega)]{\abs{\uu_h}^2 - I_h\big[\abs{\uu_h}^2\big]}
\lesssim h^2 \norm[\LL^2(\Omega)]{\Grad\uu_h}^2.
\end{align*}
On the other hand, since $\{s_h\}$ is uniformly bounded in $L^{\infty}(\Omega)$,
we infer that
\begin{equation*}
\begin{split}
\norm[L^1(\Omega)]{I_h\big[\abs{\uu_h}^2 - \abs{s_h}^2\big]}
& = \norm[L^1(\Omega)]{I_h\big[\abs{s_h}^2(\abs{\nn_h}^2 - 1)\big]} \\
& \le \norm[L^{\infty}(\Omega)]{s_h}^2 \norm[L^1(\Omega)]{I_h\big[\abs{\nn_h}^2 - 1\big]}
\le \eps\norm[L^{\infty}(\Omega)]{s_h}^2\to0,
\end{split}
\end{equation*}
as $\eps\to0$.
As $\abs{s_h} \to \abs{s}$ and $\abs{\uu_h} \to \abs{\uu}$ a.e.\ in $\Omega$,
we conclude that $\abs{s} = \abs{\uu}$ a.e.\ in $\Omega$.

\medskip
\emph{Step~3: Convergence of $\{\nn_h\}$}.
We now define $\nn : \Omega \to \R^3$
as $\nn := s^{-1}\uu$ in $\Omega\setminus\Sigma$
and as an arbitrary unit vector in $\Sigma$. Step 2 implies,
by construction, that $\abs{\nn} = 1$ a.e.\ in $\Omega$.
This shows that $(s,\nn,\uu)$ satisfies the structural condition~\eqref{eq:structural},
i.e., $(s,\nn,\uu) \in \A$.

We now observe that $s(x) \neq 0$ for a.e.\ $x \in \Omega\setminus\Sigma$
by definiton of $\Sigma$.
Since $s_h(x) \to s(x)$ as $h\to0$, if $h$ is sufficiently small
(depending on $x$), then $s_h(x) \neq 0$ is valid. Consequently,

\begin{equation*}
\nn_h(x)
= \frac{s_h(x) \nn_h(x)}{s_h(x)}
\to \frac{\uu(x)}{s(x)}
= \nn(x),
\end{equation*}
i.e., $\nn_h \to \nn$ pointwise a.e.\ in $\Omega\setminus\Sigma$.
Since $\{\nn_h\}$ is uniformly bounded in $\LL^{\infty}(\Omega)$,
the Lebesgue dominated convergence theorem
yields $\nn_h \to \nn$ strongly in $\LL^2(\Omega\setminus\Sigma)$.

Finally,
the $L^2$-differentiability
of $\nn$ 
and the orthogonal decomposition of $\Grad\uu$,
both valid a.e.\ in $\Omega\setminus\Sigma$,
follow from Proposition~\ref{prop:L2diff} (orthogonal decomposition).
This concludes the proof.
\end{proof}

We are now in the position to prove the lim-inf inequality.

\begin{proof}[\bf Proof of Theorem~\ref{thm:gamma_convergence}{\rm(ii)}]
The sequence $\{(s_h,\nn_h,\uu_h)\} \subset \A_{h,\eps}(g_h,\rr_h)$
satisfies the assumptions of Lemma~\ref{lem:limits} (characterization of limits).
Hence, we can apply it to obtain subsequences (not relabeled) converging
to the respective limits $(s,\nn,\uu) \in \A(g,\rr)$.
Moreover, since also
the sequences $\{\nn_h \otimes \grad s_h\}$ and $\{ s_h \Grad \nn_h \}$
are uniformly bounded in $\LL^2(\Omega)$,
there exist subsequences (not relabeled) and functions $\MM, \NN$ in $\LL^2(\Omega)$
such that
$\nn_h \otimes \grad s_h \weakto \MM$ and 
$s_h \Grad \nn_h \weakto \NN$
weakly in $\LL^2(\Omega)$.
Combining the equality $s_h \Grad \nn_h = \Grad (s_h \nn_h) - \nn_h \otimes \grad s_h$, which is valid in every element of $\mesh$, with $s_h \nn_h \weakto \uu$ weakly in $\HH^1(\Omega)$, helps identify the limits $\NN = \Grad\uu - \MM$.

Let $\PPhi \in \CC^{\infty}_c(\Omega\setminus\Sigma)$ be an arbitrary $d\times d$ tensor field. We can thus write
\begin{equation*}
\inner[\Omega\setminus\Sigma]{\nn_h \otimes \grad s_h - \nn \otimes \grad s}{\PPhi}
= \inner[\Omega\setminus\Sigma]{(\nn_h - \nn) \otimes \grad s_h}{\PPhi}
+ \inner[\Omega\setminus\Sigma]{\nn \otimes (\grad s_h - \grad s)}{\PPhi}.
\end{equation*}
We note that $\nn_h \to \nn$ strongly in $L^2(\Omega\setminus\Sigma)$ implies
\begin{equation*}
\inner[\Omega\setminus\Sigma]{(\nn_h - \nn) \otimes \grad s_h}{\PPhi}
\leq \norm[\LL^2(\Omega\setminus\Sigma)]{\nn_h - \nn}
\norm[\LL^2(\Omega)]{\grad s_h}
\norm[\LL^{\infty}(\Omega\setminus\Sigma)]{\PPhi}
\to 0,
\end{equation*}
whereas $s_h \weakto s$ weakly in $H^1(\Omega)$ yields
\begin{equation*}
\inner[\Omega\setminus\Sigma]{\nn \otimes (\grad s_h - \grad s)}{\PPhi}
\to 0.
\end{equation*}
Hence, we infer that
\begin{equation*}
\inner[\Omega\setminus\Sigma]{\nn_h \otimes \grad s_h - \nn \otimes \grad s}{\PPhi}
\to 0,
\end{equation*}
whence $\nn_h \otimes \grad s_h \weakto \nn \otimes \grad s$ weakly in $\LL^2(\Omega\setminus\Sigma)$.
This in turn identifies the limit
$\MM = \nn \otimes \grad s$, and gives thus the identity
$\NN = \Grad \uu - \nn \otimes \grad s$ a.e.\ in $\Omega\setminus\Sigma$.
We deduce that $\Grad \nn = \NN/s$,
where $\Grad\nn$ is understood in the $L^2$-sense according to Proposition
\ref{prop:L2diff}.
Exploiting the fact that norms are weakly lower semicontinuous, along with
$\abs{\nn \otimes \grad s}^2 = \abs{\grad s}^2$ a.e.\ in $\Omega \setminus \Sigma$,
and $\grad s = \0$ a.e.\ in $\Sigma$,
it holds that
\begin{equation*}
\begin{split}
\liminf_{h \to 0} E_1^h[s_h,\nn_h]
& =
\liminf_{h \to 0} \Big\{\frac{\kappa}{2} \norm[\LL^2(\Omega)]{\nn_h\otimes\grad s_h}^2
+ \frac{1}{2} \norm[\LL^2(\Omega)]{s_h\Grad\nn_h}^2 \Big\}\\
& \geq
\liminf_{h \to 0} \Big\{\frac{\kappa}{2} \norm[\LL^2(\Omega\setminus\Sigma)]{\nn_h\otimes\grad s_h}^2
+ \frac{1}{2} \norm[\LL^2(\Omega\setminus\Sigma)]{s_h\Grad\nn_h}^2 \Big\}\\
& \geq
\frac{\kappa}{2} \norm[\LL^2(\Omega\setminus\Sigma)]{\nn\otimes\grad s}^2
+ \frac{1}{2} \norm[\LL^2(\Omega\setminus\Sigma)]{s\Grad\nn}^2
= E_1[s,\nn].
\end{split}
\end{equation*}
Since $s_h\to s$ a.e.\ in $\Omega$ and $\psi$ is continuous,
$\psi(s_h)\to \psi(s)$ a.e.\ in $\Omega$.
The Fatou lemma yields
\begin{equation*}
E_2[s]
= \int_\Omega \psi(s)
= \int_\Omega \lim_{h \to 0}\psi(s_h)
\le \liminf_{h \to 0} \int_\Omega \psi(s_h)
= \liminf_{h \to 0} E_2^h[s].
\end{equation*}
Altogether, we thus obtain the lim-inf inequality~\eqref{eq:liminf_inequality},
namely $E[s,\nn]\le \liminf_{h \to 0} E[s_h,\nn_h]$.
This finishes the proof.
\end{proof}

\subsection{Properties of the numerical scheme}\label{sec:prop-scheme}

To start with,
we prove well-posedness and stability of Algorithm~\ref{alg:gradient_flow}.

\begin{proof}[\bf Proof of Proposition~\ref{prop:scheme}]
Let $i \in \N_0$ and $\ell \in \N_0$.
For fixed $s_h^i \in V_h$
(resp., $\nn_h^{i+1} \in \VVV_h$),
the left-hand side of~\eqref{eq:gradientflow1}
(resp., of~\eqref{eq:gradientflow3})
is a coercive and continuous bilinear form on $\VVV_{h,D}$
(resp., on $V_{h,D}$).
Therefore, the variational problem admits a unique solution
$\tt_h^{i,\ell} \in \dts[\nn_h^{i,\ell}]$
(resp., $s_h^{i+1} \in V_h$).
This shows part~(i) and~(iii) of Proposition~\ref{prop:scheme}.

Choosing the test function
$\pphi_h=\tau_{\nn} \, \tt_h^{i,\ell}
 = \nn_h^{i,\ell+1} - \nn_h^{i,\ell}\in \dts[\nn_h^{i,\ell}]$
in~\eqref{eq:gradientflow1}
yields
\begin{equation*}
\tau_{\nn} \norm[*]{\tt_h^{i,\ell}}^2
+ \kappa \inner[\Omega]{\nn_h^{i,\ell+1} \otimes \grad s_h^i}{(\nn_h^{i,\ell+1} - \nn_h^{i,\ell}) \otimes \grad s_h^i}
+ \inner[\Omega]{s_h^i \Grad \nn_h^{i,\ell+1}}{s_h^i \Grad (\nn_h^{i,\ell+1} - \nn_h^{i,\ell})}
= 0.
\end{equation*}
Using the identity $2a(a-b) = a^2 - b^2 + (a-b)^2$, valid for all $a,b\in \R$,
we obtain
\begin{equation*}
\begin{split}
\tau_{\nn} \norm[*]{\tt_h^{i,\ell}}^2 &+ \frac{\kappa}{2} \norm[\LL^2(\Omega)]{\nn_h^{i,\ell+1} \otimes \grad s_h^i}^2
- \frac{\kappa}{2} \norm[\LL^2(\Omega)]{\nn_h^{i,\ell} \otimes \grad s_h^i}^2
+ \frac{\kappa}{2} \norm[\LL^2(\Omega)]{(\nn_h^{i,\ell+1} - \nn_h^{i,\ell}) \otimes \grad s_h^i}^2 \\
& 
+ \frac{1}{2} \norm[\LL^2(\Omega)]{s_h^i \Grad \nn_h^{i,\ell+1}}^2
- \frac{1}{2} \norm[\LL^2(\Omega)]{s_h^i \Grad \nn_h^{i,\ell}}^2
+ \frac{1}{2} \norm[\LL^2(\Omega)]{s_h^i \Grad(\nn_h^{i,\ell+1} - \nn_h^{i,\ell})}^2
= 0,
\end{split}
\end{equation*}
which can be rewritten in more compact form as
\begin{equation} \label{eq:scheme_aux1}
E_1^h[s_h^i,\nn_h^{i,\ell+1}]
- E_1^h[s_h^i,\nn_h^{i,\ell}]
+ \tau_{\nn} \norm[*]{\tt_h^{i,\ell}}^2
+ \tau_{\nn}^2 \, E_1^h[s_h^i,\tt_h^{i,\ell}]
= 0.
\end{equation}
In particular, $E_1^h[s_h^i,\nn_h^{i,\ell+1}] \le E_1^h[s_h^i,\nn_h^{i,\ell}]$
is valid.
Since $E_1^h[s_h^i,\nn_h^{i,\ell}] \ge 0$ for all $i \in \N_0$,
the sequence $\{E_1^h[s_h^i,\nn_h^{i,\ell}]\}_{\ell\in\N_0}$ is convergent
(as it is monotonically decreasing and bounded from below).
In particular, it is a Cauchy sequence, which entails that
the stopping criterion~\eqref{eq:gradientflow2_stopping} is met in a finite number of iterations.
This shows part~(ii) of the proposition.

Let $\ell_i \in \N_0$ be the smallest integer for which the stopping criterion~\eqref{eq:gradientflow2_stopping} is satisfied.
Recall that $\nn_h^{i+1}=\nn_h^{i,\ell_i+1}$ and $\nn_h^i=\nn_h^{i,0}$.
Summation of~\eqref{eq:scheme_aux1} over $\ell=0,\dots,\ell_i$ yields
\begin{equation} \label{eq:scheme_aux2}
E_1^h[s_h^i,\nn_h^{i+1}]
- E_1^h[s_h^i,\nn_h^i]
+ \tau_{\nn} \sum_{\ell =0}^{\ell_i} \norm[*]{\tt_h^{i,\ell}}^2
+ \tau_{\nn}^2 \sum_{\ell =0}^{\ell_i} E_1^h[s_h^i,\tt_h^{i,\ell}]
= 0.
\end{equation}
Choosing the test function $w_h=\tau_s d_t s_h^{i+1} = s_h^{i+1} - s_h^i \in V_{{h,D}}$ in~\eqref{eq:gradientflow3}
and performing the same algebraic computation as above,
we arrive at
\begin{equation*}
\begin{split}
E_1^h[s_h^{i+1},\nn_h^{i+1}]
& - E_1^h[s_h^i,\nn_h^{i+1}] \\
& + \tau_s \norm[L^2(\Omega)]{d_t s_h^{i+1}}^2
+ \tau_s^2 \, E_1^h[d_t s_h^{i+1},\nn_h^{i+1}]
+ \inner[\Omega]{\psi_c'(s_h^{i+1}) - \psi_e'(s_h^i)}{s_h^{i+1}-s_h^i}
= 0.
\end{split}
\end{equation*}
Applying~\cite[Lemma~4.1]{nwz2017}, 
which yields the inequality
\begin{equation*}
E_2^h[s_h^{i+1}]
- E_2^h[s_h^i]
\leq
\inner[\Omega]{\psi_c'(s_h^{i+1}) - \psi_e'(s_h^i)}{s_h^{i+1}-s_h^i},
\end{equation*}
we obtain
\begin{equation*}
E_1^h[s_h^{i+1},\nn_h^{i+1}]
- E_1^h[s_h^i,\nn_h^{i+1}]
+ \tau_s \, \norm[L^2(\Omega)]{d_t s_h^{i+1}}^2
+ \tau_s^{2} \, E_1^h[d_t s_h^{i+1},\nn_h^{i+1}]
+ E_2^h[s_h^{i+1}]
- E_2^h[s_h^i]
\leq 0.
\end{equation*}
Adding the latter with~\eqref{eq:scheme_aux2}, and exploiting cancellation of
$E_1^h[s_h^i,\nn_h^{i+1}]$, we deduce
\begin{equation}\label{eq:energy_decrease_2}
\begin{split}
E^h[s_h^{i+1},\nn_h^{i+1}]
- E^h[s_h^i,\nn_h^i]
 \leq
& - \tau_{\nn} \sum_{\ell=0}^{\ell_i}\norm[*]{\tt_h^{i,\ell}}^2
- \tau_{\nn}^2 \sum_{\ell=0}^{\ell_i}E_1^h[s_h^i,\tt_h^{i,\ell}]
\\
& 
- \tau_s \norm[L^2(\Omega)]{d_t s_h^{i+1}}^2
- \tau_s^2 E_1^h[d_t s_h^{i+1},\nn_h^{i+1}] \le 0.
\end{split}
\end{equation}
This shows~\eqref{eq:energy_decrease} and concludes the proof.
\end{proof}

We recall that Algorithm \ref{alg:gradient_flow} does not enforce the unit-length constraint of the director field $\nn_h^j$. We finish this paper with a proof that violation of such constraint is controlled by $\tau_{\nn}$ and that $\|\nn_h^j\|_{\LL^\infty(\Omega)}$ is uniformly bounded provided the parameters $h$ and $\tau_{\nn}$ are suitably chosen.

\begin{proof}[\bf Proof of Proposition~\ref{prop:properties_n}]
Let $j \geq 1$.
Summation of~\eqref{eq:energy_decrease_2} over $i=0, \dots, j-1$ yields

\begin{equation}\label{eq:aux123}
E^h[s_h^j,\nn_h^j]
  + \tau_{\nn}  \sum_{i=0}^{j-1}\sum_{\ell=0}^{\ell_i} \norm[*]{\tt_h^{i,\ell}}^2
  \le  E^h[s_h^0,\nn_h^0].
\end{equation}
Moreover, the tangential update $\tt_h^{i,\ell}(z)$ is perpendicular to $\nn_h^{i,\ell}(z)$ for all $z \in \nodes$, whence $\nn_h^{i,\ell+1}(z) = \nn_h^{i,\ell}(z) + \tau_{\nn}\tt_h^{i,\ell}(z)$ satisfies
$
\abs{\nn_h^{i,\ell+1}(z)}^2 = \abs{\nn_h^{i,\ell}(z)}^2 + \tau_{\nn}^2\abs{\tt_h^{i,\ell}(z)}^2.
$
Iterating in $\ell$ and $i$ gives
\begin{equation*}
\abs{\nn_h^j(z)}^2
= \abs{\nn_h^0(z)}^2
+ \tau_{\nn}^2 \sum_{i=0}^{j-1} \sum_{\ell=0}^{\ell_i}\abs{\tt_h^{i,\ell}(z)}^2
= 1
+ \tau_{\nn}^2 \sum_{i=0}^{j-1} \sum_{\ell=0}^{\ell_i}\abs{\tt_h^{i,\ell}(z)}^2
\ge 1.
\end{equation*}
Then,
using the equivalence of the $L^p$-norm of a discrete function with
the weighted $\ell^p$-norm of the vector collecting its nodal values
(see, e.g., \cite[Lemma~3.4]{bartels2015}),
for $h_z$ being the diameter of the nodal patch associated with $z\in\nodes$,
we see that
\begin{equation*}
\norm[L^1(\Omega)]{I_h[\abs{\nn_h^j}^2]-1}
\lesssim
\sum_{z\in\nodes} h_z^d \big( \abs{\nn_h^j(z)}^2 -1 \big)
\leq
\tau_{\nn}^2 \sum_{z\in\nodes} h_z^d \sum_{i=0}^{j-1} \sum_{\ell=0}^{\ell_i}\abs{\tt_h^{i,\ell}(z)}^2
\lesssim \tau_{\nn}^2 \sum_{i=0}^{j-1} \sum_{\ell=0}^{\ell_i}\norm[\LL^{2}(\Omega)]{\tt_h^{i,\ell}}^2.
\end{equation*}
Combining \eqref{eq:metric_norm} with \eqref{eq:aux123} leads to
\begin{equation*}
\norm[L^1(\Omega)]{I_h[\abs{\nn_h^j}^2]-1}
\lesssim 
C_* \tau_{\nn}^2\sum_{i=0}^{j-1}\sum_{\ell=0}^{\ell_i}\norm[*]{\tt_h^{i,\ell}}^2
\le
C_* \tau_{\nn} \, E^h[s_h^0,\nn_h^0],
\end{equation*}
which turns out to be~\eqref{eq:L1error}.

It remains to estimate $\norm[\LL^{\infty}(\Omega)]{\nn_h^j}$.
Let us consider first the weighted $H^1$-metric~\eqref{eq:metricH1_weighted}.
Using a global inverse estimate (see, e.g., \cite[Remark~3.8]{bartels2015})
and the Poincar\'e inequality, we obtain
\begin{equation*}
\begin{split}
\norm[\LL^{\infty}(\Omega)]{\nn_h^j}^2 - 1
& = \max_{z \in \nodes} \abs{\nn_h^j(z)}^2 - 1
\leq
\tau_{\nn}^2 \sum_{i=0}^{j-1} \sum_{\ell=0}^{\ell_i}\max_{z \in \nodes} \abs{\tt_h^{i,\ell}(z)}^2 \\
& \lesssim
\tau_{\nn}^2 \sum_{i=0}^{j-1}\sum_{\ell=0}^{\ell_i}\norm[\LL^{\infty}(\Omega)]{\tt_h^{i,\ell}}^2
\lesssim
\tau_{\nn}^2 \, h_{\min}^{2-d} |\log h_{\min}|^2 \sum_{i=0}^{j-1}\sum_{\ell=0}^{\ell_i} \norm[\HH^1(\Omega)]{\tt_h^{i,\ell}}^2 \\
& \lesssim
\tau_{\nn}^2 \, h_{\min}^{2-d-\alpha} |\log h_{\min}|^2 \sum_{i=0}^{j-1}\sum_{\ell=0}^{\ell_i} \norm[\LL^2(\Omega)]{h^{\alpha/2} \Grad \tt_h^{i,\ell}}^2 \\
& \le
\tau_{\nn} \, h_{\min}^{2-d-\alpha} |\log h_{\min}|^2 \, E^h[s_h^0,\nn_h^0].
\end{split}
\end{equation*}
Therefore, \eqref{eq:Linfty_bound} is satisfied if $\tau_{\nn} \, h_{\min}^{2-d-\alpha} |\log h_{\min}|^2 \leq C^*$ with $C^*$ arbitrary.
For the $L^2$-metric~\eqref{eq:metricL2}, the result follows analogously,
provided that $\tau_{\nn} \, h_{\min}^{-d}\leq C^*$.
\end{proof}

\section*{Acknowledgments}

This project started while RHN and MR were in residence at the Institute for Computational and Experimental Research in Mathematics
(ICERM) during the workshop \emph{Numerical Methods and New Perspectives for Extended Liquid Crystalline Systems} in 2019 (grant DMS-1439786).
MR acknowledges partial support of the Austrian Science Fund (FWF) through the special research program
\emph{Taming complexity in partial differential systems} (grant F65)
and of the Erwin Schr\"odinger International Institute for Mathematics and Physics (ESI),
given during the workshop \emph{New Trends in the Variational Modeling and Simulation of Liquid Crystals}.
RHN and SY acknowledge partial support of the National Science Foundation
(grant DMS--1908267).

\bibliographystyle{acm}
\bibliography{ref}
\end{document}